\documentclass[EJP]{ejpecp} % replace ECP by EJP if needed.  

\usepackage[applemac]{inputenc}

\usepackage{enumerate}  % uncomment to use this package
\usepackage{amsfonts,mathrsfs,bbm}

\usepackage{cleveref}
  \crefname{theorem}{Theorem}{Theorems}
  \crefname{lemma}{Lemma}{Lemmas}
  \crefname{remark}{Remark}{Remarks}
  \crefname{proposition}{Proposition}{Propositions}
  \crefname{definition}{Definition}{Definitions}
  \crefname{corollary}{Corollary}{Corollaries}
  \crefname{section}{Section}{Sections}
  \crefname{figure}{Figure}{Figures}

\newcommand \exc {\textnormal{exc}}
\newcommand \br {\textnormal{br}}

\newcommand \E {\mathbb E}
\newcommand\D {\mathbb D}
\newcommand\R {\mathbb R}
\newcommand\X {X^\textnormal{exc}}
\newcommand\Ta {\mathcal{T}_{\alpha}}
\newcommand\GW {\mathsf{GW}}

\newcommand\oD {{ \Delta^*}}
\newcommand\La {\mathscr{L}_\alpha}
\renewcommand\H {H^\textnormal{exc}}
\def\llbracket{[\hspace{-.10em} [ }
\def\rrbracket{ ] \hspace{-.10em}]}
\newcommand\z {\zeta}

\newcommand\W {\mathsf{W}}
\renewcommand\P {\mathbb{P}}
\renewcommand\La {\mathscr{L}_\alpha}
\newcommand{\N}{\mathbb N}

\newcommand\Es[1]{\mathbb{E}\left[#1\right]}
\renewcommand\Pr[1]{\mathbb{P}\left(#1\right)}
\newcommand{\op}[1]{\operatorname{#1 }}
\newcommand{\T}{\mathcal{T}}

\SHORTTITLE{Random stable looptrees} 

\TITLE{Random stable looptrees\thanks{This work is partially supported by the French ÒAgence Nationale de la RechercheÓ  ANR-08-BLAN-0190.}} % \thanks is optional. Insert line breaks with \\

\AUTHORS{%
  Nicolas~Curien\footnote{CNRS and UniversitŽ Paris 6, France.
    \EMAIL{nicolas.curien@gmail.com}}
  \and %% remove this line and below if single author
  Igor~Kortchemski\footnote{DMA, ƒcole Normale SupŽrieure,
    France. \EMAIL{igor.kortchemski@normalesup.org} }}%AUTHORS
 % Separate items with ;

\AMSSUBJ{60F17; 60G52} % Edit. Separate items with ;
\AMSSUBJSECONDARY{05C80} % Optional, separate items with ;

\SUBMITTED{April 9, 2013} % Edit.
\ACCEPTED{November 10, 2014} % Edit.

\VOLUME{19}
\YEAR{2014}
\PAPERNUM{108}
\DOI{v19-2732}

\ABSTRACT{We introduce a class of random compact metric spaces $\mathscr{L}_{\alpha}$ indexed by  $\alpha~\in(1,2)$ and which we call stable looptrees. They are made of a collection of random loops glued together along a tree structure, and can informally be viewed as dual graphs of  $\alpha$-stable LŽvy trees. We study their properties and prove in particular that the Hausdorff dimension of $ \mathscr{L}_{\alpha}$ is almost surely equal to $\alpha$. We also show that stable looptrees are universal scaling limits, for the Gromov--Hausdorff topology, of various combinatorial models. In a companion paper, we prove that the stable looptree of parameter $ \frac{3}{2}$ is the scaling limit of cluster boundaries in critical site-percolation on large random triangulations.
}

\begin{document}

%%                                                               %%
%% No need for \maketitle.                                       %%
%%                                                               %%

%%                                                               %%
%% Please replace what follows by the body of your article       %%
%% (up to the bibliography):                                     %%
%%                                                               %%

\vfill

 \begin{figure}[!h]
 \begin{center}
  \includegraphics[width=0.8 \linewidth]{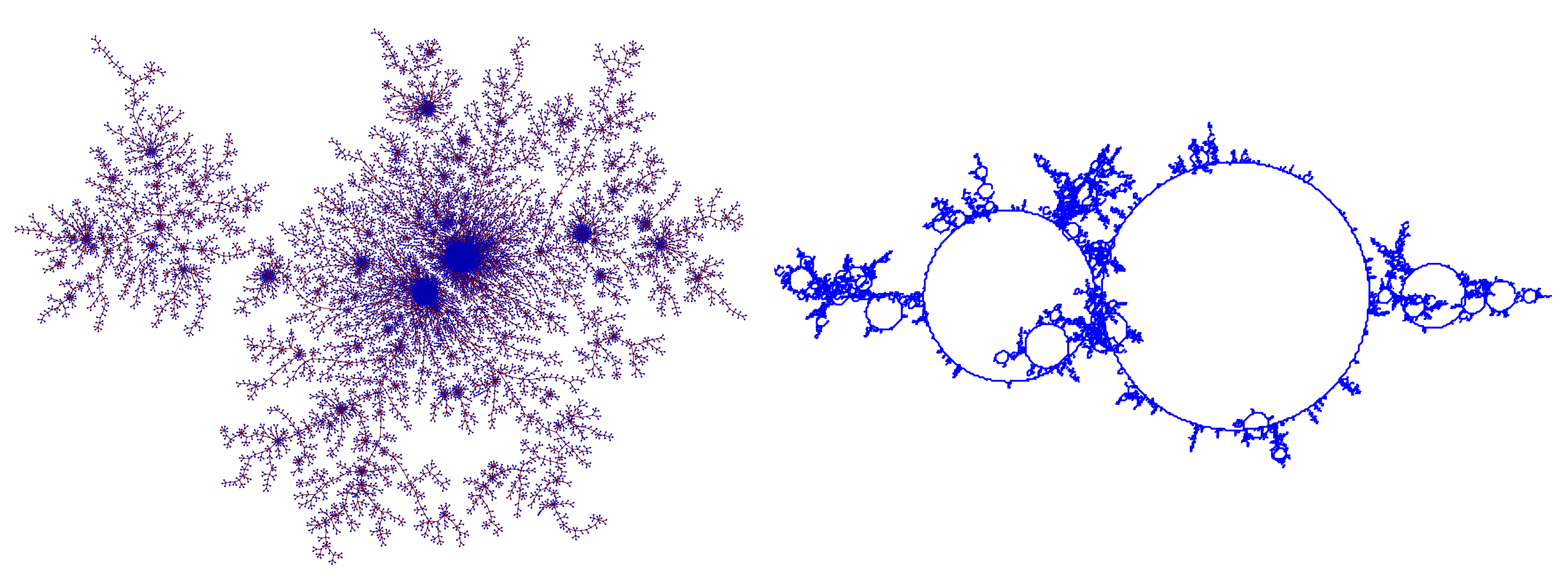}
 \caption{ \label{fig:frerots} An $ \alpha=1.1$ stable tree, and its {associated}
 looptree {$\mathscr{L}_{1.1}$}, embedded non isometrically in the plane
 (this embedding of $\mathscr{L}_{1.1}$ contains intersecting loops, even though they are disjoint in the metric space).}
 \end{center}
 \end{figure}
 \vfill

 \clearpage

\section{Introduction}
In this paper, we introduce and study a new family $ (\mathscr{L}_\alpha)_{1 < \alpha <2}$ of random compact metric spaces which we call stable \emph{looptrees} (in short, looptrees). Informally, they are constructed from the stable tree of index $ \alpha$ introduced in \cite{DLG02,LGLJ98} by replacing each branch-point of the tree by a cycle of length proportional to the ``width'' of the branch-point   and then gluing the cycles along the tree structure (see \cref{def:looptree} below). We  study their fractal properties and calculate in particular their Hausdorff dimension. We also prove that looptrees naturally appear as scaling limits for the Gromov--Hausdorff topology of various discrete random structures, such as {Boltzmann-type random dissections which were introduced in \cite{Kor11}.}

 Perhaps more unexpectedly, looptrees appear in the study of random maps decorated with statistical physics models. More precisely, in a companion paper \cite{CKpercolooptrees}, we prove that the stable looptree of parameter $ \frac{3}{2}$ is the scaling limit of cluster boundaries in critical site-percolation on large random triangulations and on the uniform infinite planar triangulation of Angel \& Schramm \cite{AS03}. We also conjecture a more general statement for $O(n)$ models on random planar maps.

\medskip

\begin{center} In this paper $\alpha \in (1,2)$. \end{center}

 \paragraph*{Stable looptrees as limits of discrete looptrees.} In order to explain the intuition leading to the definition of stable looptrees, we first introduce them as limits of random discrete graphs (even though they will be defined later  without any reference to discrete objects). To this end, with every rooted oriented tree (or plane tree) $ \tau$, we associate a graph denoted by $ \mathsf{Loop}( \tau)$ and constructed by replacing each vertex $u \in \tau$ by a discrete cycle of length given by the degree of $u$ in $ \tau$ (i.e. number of neighbors of $u$) and gluing all these cycles according to the tree structure provided by $\tau$, see \cref{fig:loop} {(by discrete cycle of length $k$, we mean a graph on $k$ vertices $v_{1}, \ldots,v_{k}$ with edges $v_{1}v_{2}, \ldots, v_{k-1}v_{k},v_{k}v_{1}$)}. We endow $ \mathsf{Loop}( \tau)$ with the graph distance (every edge has unit length).

 \begin{figure}[!h]
 \begin{center}
 \includegraphics[width=0.7  \linewidth]{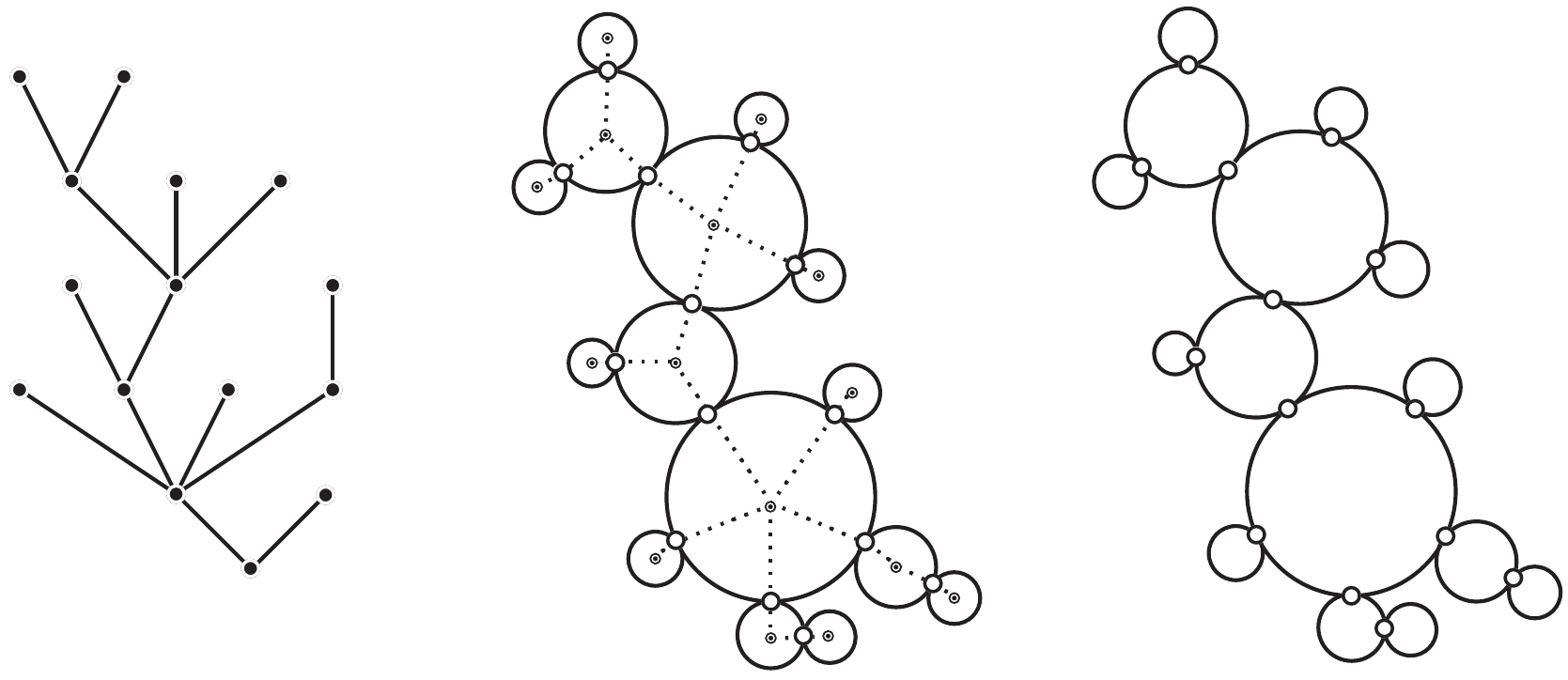}
 \caption{ \label{fig:loop}A discrete tree $\tau$ and its associated discrete looptree $ \mathsf{Loop}( \tau)$.}
 \end{center}
 \end{figure}
 
Fix $\alpha \in (1,2)$ and let $\tau_{n}$ be a Galton--Watson tree conditioned on having $n$ vertices, whose offspring distribution $\mu$ is critical and satisfies $\mu([k, \infty)) \sim |\Gamma(1- \alpha)|^{-1} \cdot k^{-\alpha}$ as $k \rightarrow \infty$.  The stable looptree $ \mathscr{L}_{\alpha}$ then appears (\cref{thm:continuity}) as the scaling limit in distribution for the Gromov--Hausdorff topology of discrete looptrees $ \mathsf{Loop}( \tau_{n})$:  \begin{eqnarray} n^{-1/\alpha} \cdot \mathsf{Loop}( \tau_{n}) & \quad \xrightarrow[n\to\infty]{(d)} \quad & \mathscr{L}_{\alpha},  \label{eq:invprinc} \end{eqnarray}
where $c \cdot M$ stands for the metric space obtained from $M$ by multiplying all distances by $c >0$.   Recall that the Gromov--Hausdorff topology gives a sense to convergence of (isometry classes) of compact metric spaces, see \cref{sec:limitcases} below for the definition.

It is known that the random trees $ \tau_{n}$ converge, after suitable scaling, towards the so-called stable tree $ \mathcal{T}_{\alpha}$ of index $ \alpha$ (see \cite{Du03,DLG02,LGLJ98}). It thus seems natural to try to define $ \mathscr{L}_{\alpha}$ directly from $ \mathcal{T}_{\alpha}$ by mimicking the discrete setting (see \cref{fig:frerots}). However this construction is not straightforward since the countable collection of loops of $ \mathscr{L}_{\alpha}$ does not form a compact metric space:  one has to take its closure. In particular, two different cycles of $ \mathscr{L}_{\alpha}$ never share a common point.  To overcome these difficulties, we define $ \mathscr{L}_{\alpha}$ by using the excursion $X^{\exc,(\alpha)}$ of an $ \alpha$-stable spectrally positive LŽvy process (which also codes $ \mathcal{T}_{\alpha}$).

  \paragraph*{Properties of stable looptrees.}Stable looptrees possess a fractal structure whose dimension is identified by the following theorem:  \begin{theorem}[Dimension] \label{thm:dimension} For every $\alpha \in (1,2)$, almost surely, $ \mathscr{L}_{\alpha}$ is a random compact metric space of Hausdorff dimension $\alpha$. \end{theorem}
The proof of this theorem uses fine properties of the excursion $ X^{\exc,(\alpha)}$.
We also prove that the family of stable looptrees interpolates 
 between the circle of unit length $ \mathcal{C}_{1}:=(2\pi)^{-1} \cdot\mathbb{S}_{1}$ and the $2$-stable tree $ \mathcal{T}_{2}$ which is the Brownian Continuum Random Tree introduced by Aldous \cite{Ald93} (up to a constant multiplicative factor). 
 
\begin{theorem}[Interpolation loop-tree] \label{thm:1and2} The following two convergences hold in distribution for the Gromov--Hausdorff topology  \begin{eqnarray*}  (i) \quad  \mathscr{L}_{\alpha} \quad  \xrightarrow[\alpha\downarrow 1]{ (d)}  \quad  \mathcal{C}_{1},  & \mbox{}&  \qquad (ii) \quad \mathscr{L}_{\alpha}  \quad \xrightarrow[\alpha\uparrow 2]{ (d)}  \quad\frac{1}{2} \cdot \mathcal{T}_{2}.   \end{eqnarray*} 
\end{theorem}

See \cref{fig:1et2} for an illustration. The proof of $(i)$ relies on a new ``one big-jump principle'' for the normalized excursion of the $ \alpha$-stable spectrally positive LŽvy process which is of independent interest: informally, as $ \alpha \downarrow 1$, the random process $X^{\exc,(\alpha)}$ converges towards the deterministic affine function on $[0,1]$ which is equal to $1$ at time $0$ and $0$ at time $1$. We refer to \cref{prop:cvsautexc}  for a precise statement. Notice also the appearance of the factor $ \frac{1}{2}$ in $(ii)$. 
\begin{figure}[!h]
 \begin{center}
\includegraphics[width=11cm]{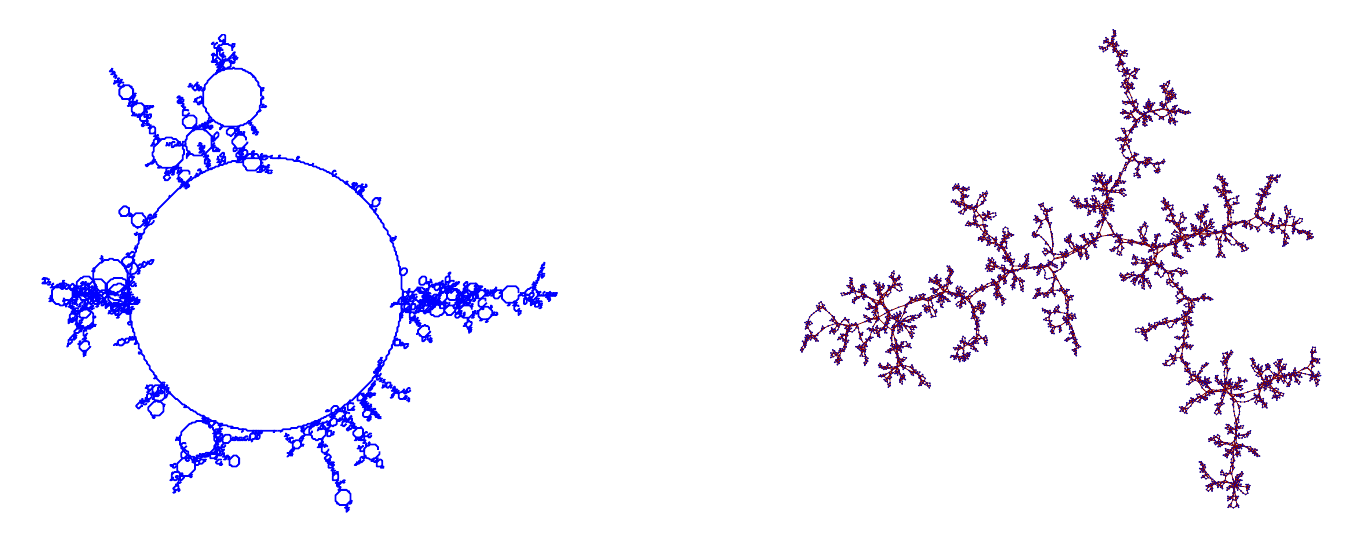}
 \caption{ \label{fig:1et2} On the left $ \mathscr{L}_{1.01}$, on the right $ \mathscr{L}_{1.9}$.}
 \end{center}
 \end{figure}

 \paragraph*{Scaling limits of Boltzmann dissections.}  Our {previously mentioned} invariance principle (\cref{thm:continuity}) also enables us to prove that stable looptrees are  scaling limits of Boltzmann dissections of \cite{Kor11}. Before giving a precise statement, we need to introduce some notation. For $n \geq 3$, let $P_{n}$ be the convex polygon inscribed in the unit disk of the complex plane whose vertices are the $n$-th roots of unity. By definition, a dissection is the union of the sides of $P_n$  and of a collection of diagonals that may intersect only at their endpoints, see \cref{fig:dual}. The faces are the connected components of the complement of the dissection in the polygon. Following \cite{Kor11},  if $ \mu=(\mu_j)_{j \geq 0}$ is a probability distribution on $\{0,2,3,4, \ldots\}$ of mean $1$, we define a
Boltzmann--type probability measure $\mathbb{P} ^ {\mu}_ {n}$ on the set of all dissections of $P_{n+1}$ by setting, for every dissection $ \omega$ of $P_{n+1}$:
$$ \mathbb{P} ^ {\mu}_ {n}(\omega)=\frac{1}{Z_n} \prod_{f \textrm{ face of } \omega}
\mu_{\deg(f)-1},$$ where $\deg(f)$ is the
degree of the face $f$, that is the number of edges in the boundary of $f$, and $Z_n$ is a normalizing constant.  Under mild assumptions on $ \mu$, this definition makes sense for every $n$ large enough. Let $ \mathcal{D}_{n}^\mu$ be a random dissection sampled according to $ \mathbb{P}_{n}^\mu$.  In \cite{Kor11}, the second author studied the {asymptotic behavior of} $ \mathcal{D}_{n}^\mu$  viewed as a random closed subset of the unit disk when $n \to \infty$ in the case where $ \mu$ has a heavy tail. Then the limiting object  (the so-called stable lamination of index $ \alpha$)  is a random compact subset of the disk {which is} the union of infinitely many non-intersecting chords and has faces of infinite degree. Its  Hausdorff dimension is a.s.\,$ 2- \alpha^{-1}$.

In this paper, instead of considering $ \mathcal{D}_{n}^\mu$ as a random compact subset of the unit disk, we view  $ \mathcal{D}_{n}^\mu$ as a metric space by endowing {the vertices of} $ \mathcal{D}_{n}^\mu$ with the graph distance (every edge of $ \mathcal{D}_{n}^\mu$ has length one). From this perspective, the scaling limit of the random Boltzmann dissections $ \mathcal{D}_{n}^\mu$ is a stable looptree (see \cref{fig:dissec+looptree}):  

\begin{corollary} \label{cor:discretencstable} Fix $ \alpha \in (1,2)$ and let $\mu$ be a probability measure supported on $\{0,2,3, \ldots\}$ of mean $1$ {such that} $\mu( [k, \infty)) \sim c \cdot k^{- \alpha}$ as $ k \to \infty$, for a certain $c>0$. Then  the following convergence holds in distribution for the Gromov--Hausdorff topology 
  \begin{eqnarray*} n^{-1/\alpha} \cdot  \mathcal{D}^\mu_{n} & \quad \xrightarrow[n\to\infty]{(d)} \quad & ( c   \mu_{0} | \Gamma(1- \alpha)|)^{-1/ \alpha}  \cdot \mathscr{L}_{\alpha}.  \end{eqnarray*}
 \end{corollary}
   
 \begin{figure}[!h]
  \begin{center}
    \includegraphics[height=3cm]{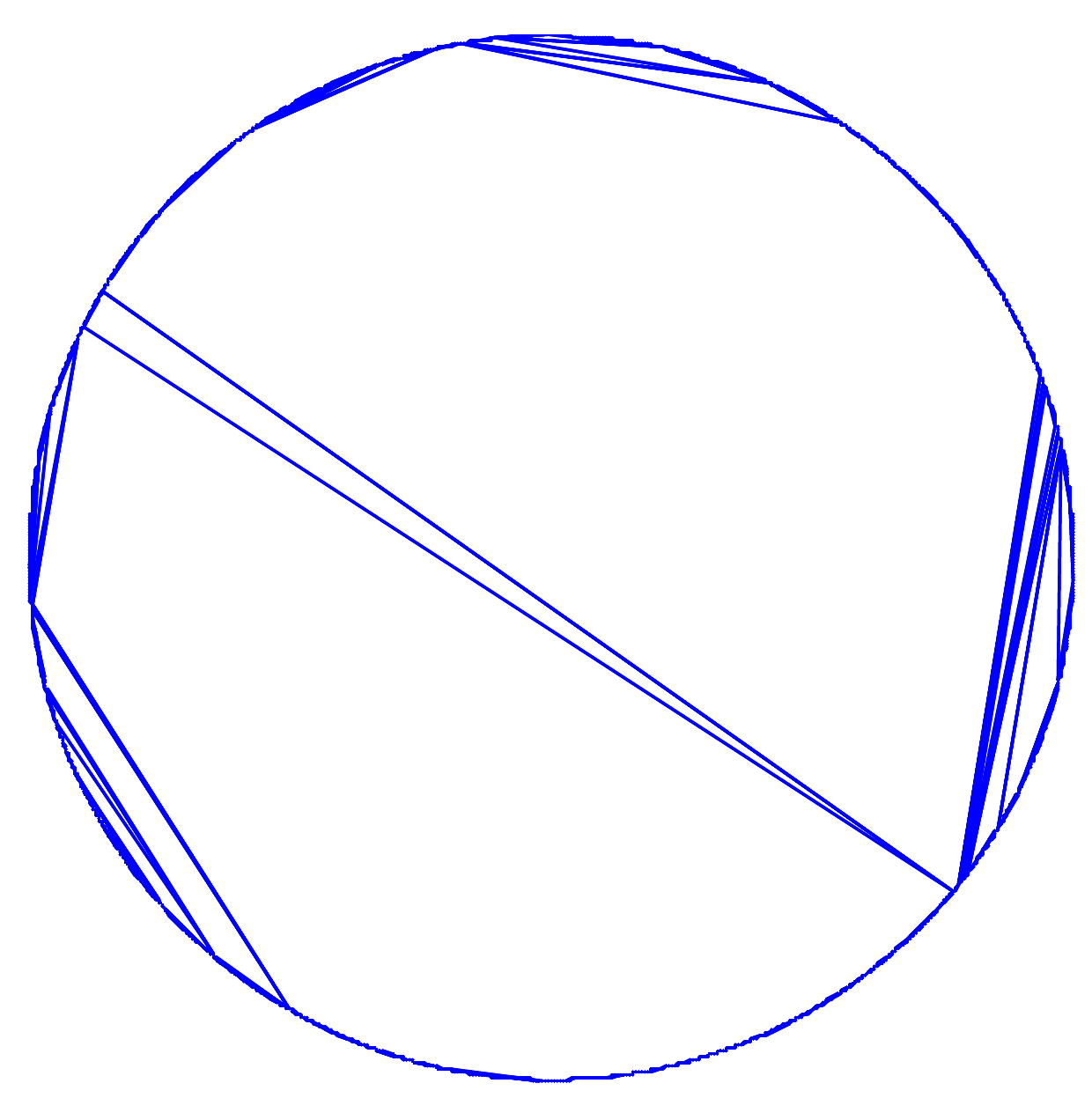} \hspace{1cm}
      \includegraphics[height=3cm]{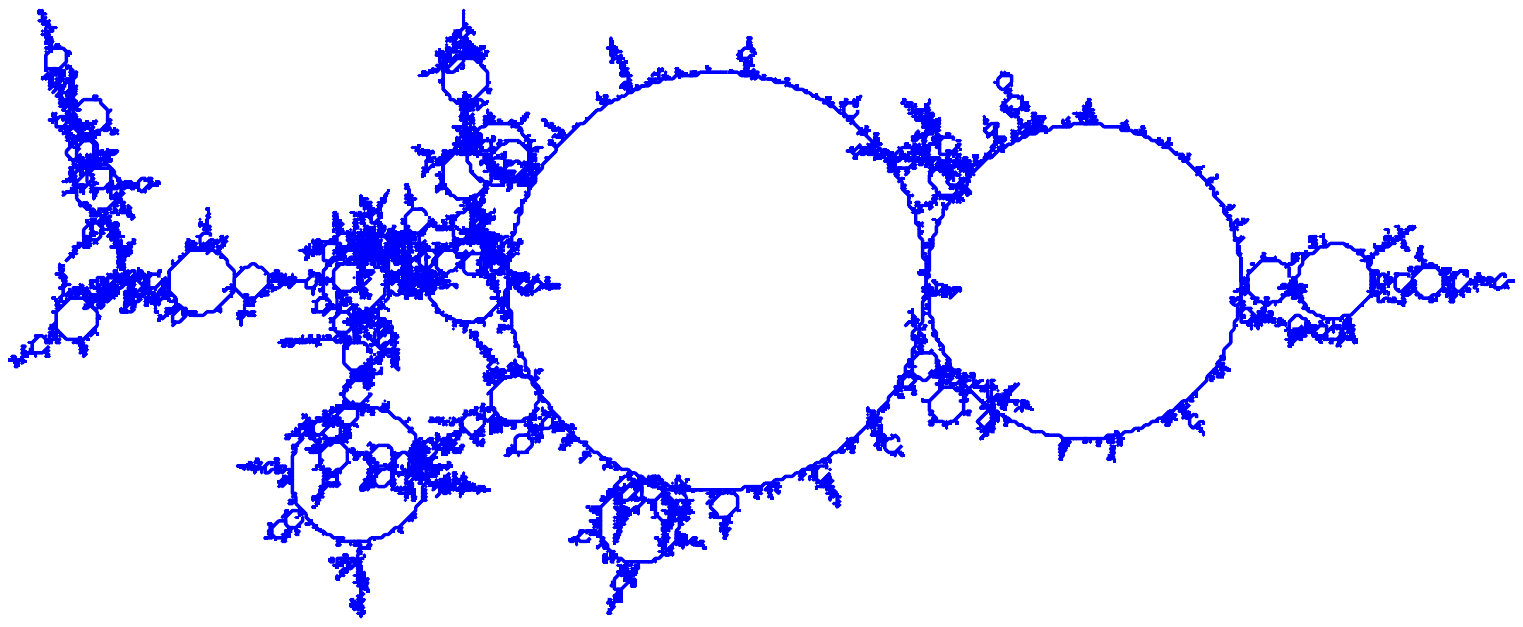}
  \caption{ \label{fig:dissec+looptree} A large dissection and a representation of its metric space.}
  \end{center}
  \end{figure}
\paragraph*{Looptrees in random planar maps.}  Another area where looptrees appear is the theory of random planar maps. The goal of this very active field is to understand large-scale properties of  planar maps or graphs, chosen uniformly in a certain class (triangulations, quadrangulations, etc.), see \cite{AS03,CS04,LGM09,LG11,Mie11}. In a companion paper \cite{CKpercolooptrees}, we prove that the scaling limit of cluster boundaries of critical site-percolation on large random triangulations and the UIPT introduced by Angel \& Schramm \cite{AS03} is $ \mathscr{L}_{3/2}$ {(by \emph{boundary of a cluster}, we mean the graph formed by the edges and vertices of a connected component which are adjacent to its exterior; see  \cite{CKpercolooptrees} for a precise definition and statement)}. We also give a precise conjecture relating the whole family of looptrees $(\mathscr{L}_{ \alpha})_{ \alpha \in (1,2)}$ to cluster boundaries of critical $O(n)$ models on random planar maps. We refer to \cite{CKpercolooptrees} for details.

\paragraph*{Looptrees in preferential attachment.} As another motivation for introducing looptrees, we mention the subsequential work  \cite{CDKM14}, which studies looptrees associated with random trees built by linear preferential attachment, also known in the literature as Barab\'asi--Albert trees or plane-oriented recursive trees.
As the number of nodes grows, it is shown in \cite {CDKM14} that these looptrees, appropriately rescaled, 
converge in the Gromov--Hausdorff sense towards a random compact metric space called the Brownian looptree, which is a quotient space of Aldous' Brownian Continuum Random Tree.

\bigskip

Finally, let us mention that  stable looptrees implicitly appear  in \cite{LGM09}, {where} Le Gall and Miermont {have} considered scaling limits of random planar maps with large faces. The limiting continuous objects (the {so-called} $\alpha$-stable maps) are constructed via a distance process {which} is closely related to looptrees. Informally, the distance process of Le Gall and Miermont is formed by a looptree $ \mathscr{L}_{\alpha}$ where the cycles support independent Brownian bridges of the corresponding lengths. However, the definition and the study of the underlying looptree structure is interesting in itself and has various applications. Even though we do not rely explicitly on the article of Le Gall and Miermont, this work would not have been possible without it.

\paragraph*{Outline.} The paper is organized as follows. In \cref{sec:construction}, we give a precise definition of $ \mathscr{L}_{\alpha}$ using the normalized excursion of the $\alpha$-stable spectrally positive LŽvy process. \cref{sec:properties} is then devoted to the study of stable looptrees, and in particular to the proofs of \cref{thm:dimension,thm:1and2}.  In the last section, we establish a general invariance principle concerning discrete looptrees from which \cref{cor:discretencstable} will follow.

 \section{Defining stable looptrees} \label{sec:construction}

This section is devoted to the construction of  stable looptrees using the normalized excursion of a stable LŽvy process, and to the study {of} their properties. In this section, $\alpha \in (1,2)$ is a fixed parameter.

\subsection{The normalized excursion of a stable LŽvy process}
\label{sec:normexc}

 We follow the presentation of \cite{Du03} and refer to \cite{Ber96} for the proof of the results mentioned here. 
By \emph{$ \alpha$-stable LŽvy process} we will always mean a 
stable spectrally
positive LŽvy process $X$ of index $\alpha$, normalized so that for every
$\lambda>0$
$$\E[\exp(-\lambda X_t)]=\exp(t \lambda^\alpha).$$
The process $X$ takes values in the Skorokhod space $\D(\R_+, \R)$ of
right-continuous with left limits (cˆdlˆg) real-valued functions,
endowed with the Skorokhod topology (see \cite[Chap. 3]{Bil99}). The dependence of $X$ in $\alpha$ will be implicit in this section.  Recall that $X$ enjoys the
following scaling property: For every $c>0$, the process $(c^{-1/\alpha}
X_{ct}, t \geq 0$) has the same law as $X$. Also recall that the LŽvy measure $\Pi$ of $X$ is
 \begin{eqnarray}\Pi(dr)&=&\frac{\alpha(\alpha-1)}{\Gamma(2-\alpha)} r^{-\alpha-1} 1_ {(0,\infty)}dr.   \label{eq:levymeasure} \end{eqnarray}

Following Chaumont \cite{Cha97} we define the normalized excursion of $X$ above its infimum as  the re-normalized excursion of $X$ above its infimum straddling time $1$.  More precisely, set $$\underline{g}_1=\sup \{s \leq 1; \, X_s= \inf_{[0,s]}X \} \quad  \mbox{and} \quad  \underline{d}_1=\inf\{s > 1; \, X_s=\inf_{[0,s]}X\}.$$
Note that  $X_{\underline{d}_{1}}=X_{\underline{g}_{1}}$  since a.s.\,$X$ has no jump at time $\underline{g}_{1}$ and $X$ has no negative jumps . Then the normalized excursion $X^{ \mathrm{exc}}$ of $X$ above its infimum is defined by
   \begin{eqnarray}\label{eq:Xexc}  X^{ \mathrm{exc}}_s&{=}&(\underline{d}_1-\underline{g}_1)^{-1/\alpha} (X_{\underline{g}_1+  s (\underline{d}_1-\underline{g}_1)} - X_{\underline{g}_1}) \qquad  \mbox{ for every } s \in [0,1].  \end{eqnarray}
We shall see later  in \cref{sec:absolutecontinuity} another useful description of $ \X$ using the It™ excursion measure of $X$ above its infimum. Notice that $ \X$ is a.s.\,a random cˆdlˆg function on $[0,1]$ such that $X^{ \mathrm{exc}}_{0}= \X_{1}=0$ and $ \X_{s}>0$ for every $s \in(0,1)$. If $Y$ is a cˆdlˆg function, we set $ \Delta Y_{t}= Y_{t}-Y_{t-}$, and to simplify notation, for $0 < t \leq 1$, we write $$ \Delta_t= X^{\exc}_{t}- X^{\exc}_{t-}$$ and set $ \Delta_{0}=0$ by convention.

\subsection{The stable LŽvy tree}
\label{sec:levytree}

We now discuss the construction of the $ \alpha$-stable tree $ \Ta$, which is closely related to the $ \alpha$-stable looptree. Even though it possible to define $ \La$ without mentioning  $ \Ta$, this sheds some light on the intuition hiding behind the formal definition of looptrees. 

\subsubsection{The stable height process} By the work of Le Gall \& Le Jan \cite{LGLJ98} and Duquesne \& Le Gall \cite{DLG02,DLG05}, it is known that the random excursion $X^{ \mathrm{exc}}$ encodes a random compact $ \mathbb{R}$-tree $ \mathcal{T}_{\alpha}$  called the $\alpha$-stable tree. To define $ \mathcal{T}_{\alpha}$, we need to introduce the \emph{height process} associated with $X^{ \mathrm{exc}}$. We refer to \cite{DLG02} and \cite{DLG05} for details and proofs of the assertions contained in this section. First, for $ 0 \leq s \leq t \leq 1$, set $$I_s^t=\inf_{[s,t]} \X.$$ The height process $ \H$ associated with $\X$ is defined by the approximation formula
 \begin{eqnarray*} \H_{t} & = & \lim_{ \varepsilon \to 0}  \frac{1}{\varepsilon}\int_{0}^t \mathrm{d}s\, \mathbbm{1}_{  \{\X_{s}< I_{s}^t + \varepsilon\}}, \qquad t \in [0,1],  \end{eqnarray*}where the limit exists in probability. The process $ (\H_{t})_{0 \leq t \leq 1}$ has a continuous modification, which we consider from now on. Then $\H$ satisfies $ \H_{0}= \H_{1}=0$ and $ \H_{t}>0$ for $t \in (0,1)$. It is  standard to define the $\mathbb{R}$-tree coded by $ \H$ as follows.  For every $h : [0,1]\to \mathbb{R}_{+}$  and  $ 0 \leq s,t \leq 1$, we set
  \begin{eqnarray} \label{def:dh} \mathrm{d}_{h}(s,t) = h({s})+ h({t})-2  \inf_{[ \min(s,t), \max(s,t)]} h.\end{eqnarray}
Recall that a pseudo-distance $d$ on a set $X$ is a map $d : X \times X \rightarrow \R_{+}$ such that $d(x,x)=0$ and $d(x,y) \leq d(x,z)+d(z,y)$ for every $x,y,z \in X$ (it is a distance if, in addition, $d(x,y)>0$ if $x \neq y$). It is simple to check that $ \mathrm{d}_{h}$ is a pseudo-distance on [0,1]. In the case $h= \H$, for $ x,y \in [0,1]$, set $ x\simeq y$ if $ \mathrm{d}_{ \H}(x,y)=0$. The random stable tree $ \mathcal{T}_{\alpha}$ is then defined as the quotient metric space $\big([0,1]/\simeq, \mathrm{d}_{ \H} \big)$, which indeed is a random compact $ \mathbb{R}$-tree \cite[Theorem 2.1]{DLG05}.  Let $\pi : [0,1] \to \mathcal{T}_{\alpha}$ be the canonical projection. The tree $ \mathcal{T}_{\alpha}$ has a distinguished point $\rho = \pi(0)$, called the \emph{root} or the ancestor of the tree.   If $u,v \in \mathcal{T}_{ \alpha}$, we denote by $\llbracket u,v \rrbracket$ the unique geodesic between $u$ and $v$. This allows us to define a genealogical order on $ \Ta$: For every $u, v \in \Ta$, set $u\preccurlyeq v$ if $u \in \llbracket \rho,v \rrbracket$. If $ u, v \in \Ta$, there exists a unique $z \in \Ta$ such that $ \llbracket \rho, u \rrbracket \cap \llbracket \rho, v \rrbracket = \llbracket \rho, z \rrbracket $, called the \emph{most recent common ancestor} to $u$ and $v$, and is denoted by $z= u \wedge v$. 

\subsubsection{Genealogy of $ \mathcal{T}_{\alpha}$ and $ X^{ \mathrm{exc}}$} 
 The genealogical order of $ \Ta$ can be easily recovered from $\X$ as follows. We define a partial order on $[0,1]$, still denoted by $\preccurlyeq$, which is compatible with the projection $ \pi: [0,1] \rightarrow \Ta$ by setting, for every $ s, t \in [0,1]$,
$$ s \preccurlyeq t \qquad \textrm{if} \qquad s\leq t \quad \mbox{ and }  \quad \X_{s-} \leq I_{s}^ {t},$$
where by convention $ \X_{0-}=0$.
It is a simple matter to check that $ \preccurlyeq$ is indeed a partial order which is compatible with the genealogical order on $ \mathcal{T}_{\alpha}$, {meaning that} a point $a \in \mathcal{T}_{\alpha}$ is an ancestor of $b$ {if and only if} there exist $s \preccurlyeq t \in [0,1]$ with $a=\pi(s) $ and $b=\pi(t)$.  For every $s,t \in [0,1]$, let $s \wedge t$ be the most recent common ancestor (for the relation $ \preccurlyeq$ on $[0,1]$) of $s$ and $t$. {Then} $ \pi(s \wedge t)$ {also is} the most recent common ancestor of $\pi(s)$ an $\pi(t)$ in the tree $ \mathcal{T}_{\alpha}$. 
 
We now recall several  well-known properties of $ \Ta$. By definition, the \emph{multiplicity} (or degree) of a vertex $u \in \mathcal{T}_{\alpha} $ is the number of connected components of $ \mathcal{T}_{\alpha} \backslash \{u\}$. Vertices of $\mathcal{T}_{\alpha} \backslash \{ \rho\}$ which have multiplicity $1$ are called \emph{leaves}, and those with multiplicity at least $3$ are called \emph{branch-points}. By \cite[Theorem 4.6]{DLG05}, the multiplicity of every vertex of $ \Ta$ belongs to $\{ 1,2,\infty\}$. In addition, the branch-points of $ \mathcal{T}_{\alpha}$ are in one-to-one correspondence with the jumps of $ \X$ \cite[Proposition 2]{Mie05}. More precisely, a vertex $u \in \mathcal{T}_{\alpha}$ is a branch-point if and only if there exists a unique $s \in [0,1]$ such that $u=\pi(s)$ and $\Delta X^{ \mathrm{exc}}_{s}= \Delta_{s} >0$. In this case $\Delta_{s}$ intuitively corresponds  to the ``number of children'' (although this does not formally make sense) or \emph{width} of $\pi(s)$. 

We finally introduce a last notation, which will be crucial in the definition of stable looptrees in the next section. If $s,t \in [0,1]$ and $s \preccurlyeq t$, set
$$  x^ {t}_s=I_{s}^ {t}- \X_{s-} \quad \in [0, {\Delta}_{s}].$$ 
Roughly speaking, $x^ {t}_s$ is the ``position'' of the ancestor of $\pi(t)$ among the $\Delta_{s}$ ``children'' of $ \pi(s)$. 

\subsection{Definition of stable looptrees}

Informally, the stable looptree $ \mathscr{L}_{\alpha}$ is obtained from the tree $ \mathcal{T}_{\alpha}$ by replacing every branch-point of width $x$ by a metric cycle of length $x$, and then gluing all these cycles along the tree structure of $ \mathcal{T}_{\alpha}$ (in a very similar way to the construction of discrete looptrees from discrete trees explained in the Introduction, see \cref{fig:frerots,fig:loop}). But making this construction rigorous is not so easy because there are countably many loops (non of them being adjacent).

\bigskip

\label{sec:definition}

Recall that the dependence in $\alpha$ is implicit through the process $X^{ \mathrm{exc}}$. For every $t \in [0,1]$ we equip the segment $[0, \Delta_{t}]$ with the pseudo-distance $ \delta_{t}$ defined by 
 \begin{eqnarray*} \delta_{t}(a,b) &=& \min \big\{|a-b| ,  \Delta_{t}-|a-b|\big\}, \qquad a,b \in [0, \Delta_{t}].\end{eqnarray*}
 Note that if $\Delta_{t}>0$, $([0, \Delta_{t}), \delta_{t})$ is isometric to a metric cycle of length $ \Delta_{t}$ (this cycle will be associated with the branch-point $ \pi(t)$ in the looptree $ \mathscr{L}_{\alpha}$, as promised in the previous paragraph).

For $s \leq t \in [0,1]$, we write $ s \prec t$ if $ s \preccurlyeq t$ and $s \neq t$. It is important to keep in mind that $ \prec$ \emph{does not} correspond to the strict genealogical order in $ \mathcal{T}_{\alpha}$ since there  exist $s \prec t$ with $\pi(s)=\pi(t)$. The stable looptree $ \mathscr{L}_{\alpha}$ will be defined as the quotient of $[0,1]$ by a certain pseudo-distance $d$ involving $ \X$, which we now define. First, if $ s \preccurlyeq t$, set 
 \begin{eqnarray} \label{d0} {d}_{0}(s,t) &=&   \sum_{s \prec r \preccurlyeq t} \delta_{{r}}( 0, {x_{r }^ {t}}).  \end{eqnarray}
In the last sum, only  jump times give a positive contribution, since $ \delta_{{r}}( 0, {x_{r }^ {t}})=0$ when $ \Delta_{r}=0$. Note that even if $t$ is a jump time, its contribution in \eqref{d0} is null since $\delta_{t}(0, x_{t}^{t})= 0$ and we could have summed over $s \prec r \prec t$. Deliberately, we do not allow $r=s$ in \eqref{d0}. Also, it could happen that there is no $r \in (s,t]$ such that both $ s \prec r$ and $r \preccurlyeq t$ (e.g. when $s=t$) in which case the sum \eqref{d0} is equal to zero. Heuristically, if $s \prec r \preccurlyeq t$, the term $ \delta_{{r}}( 0, {x_{r }^ {t}})$ represents the length of the portion of the path going from (the images in the looptree of) $s$ to $t$ belonging to the loop coded by the branch-point $r$ (see \cref{fig:explan}).  Then, for every $ s,t \in [0,1]$, set
  \begin{eqnarray} \label{eq:def2}{d}(s,t) &=& \delta_{{s \wedge t}}\big( {x_{s\wedge t }^ {s}}, {x_{s\wedge t }^ {t}}\big) + {d}_{0}( s \wedge t,s)+ {d}_{0}( s \wedge t,t).  \end{eqnarray}
  \begin{figure}[!h]
 \begin{center}
 \includegraphics[width= 0.75 \linewidth]{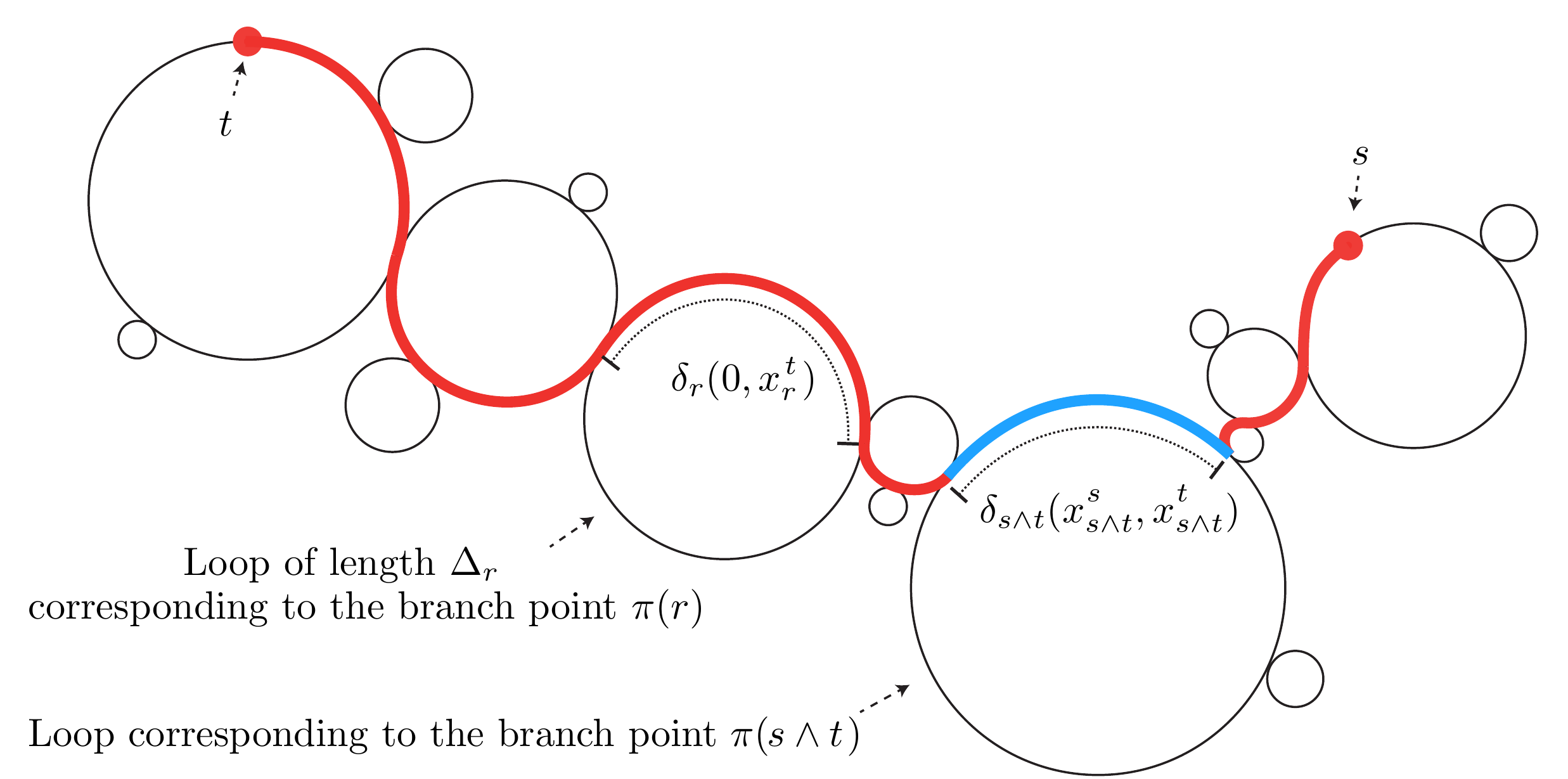}
 \caption{Illustration of the definition of $d$. The geodesic between the images of $s$ and $t$ in the looptree is in bold. Here, $ s \wedge t\prec r \prec t$. This is a simplified picture since in stable looptrees no loops are adjacent.}
 \label{fig:explan}
 \end{center}
 \end{figure}
 
 Let us give an intuitive meaning to this definition. The distance ${d}(s,t)$ contains contributions given by loops which correspond to branch-points belonging to the geodesic  $ \llbracket \pi(s ), \pi(t) \rrbracket$ in the tree: the third (respectively second) term of the right-hand side of \eqref{eq:def2} measures the contributions from branch-points belonging to the interior of $ \llbracket \pi(s \wedge t), \pi(t) \rrbracket$ (respectively  $ \llbracket \pi(s \wedge t), \pi(s) \rrbracket$), while the term $\delta_{{s \wedge t}}( {x_{s\wedge t }^ {s}}, {x_{s\wedge t }^ {t}}) $ represents the length of the portion of the path going from (the images in the looptree of) $s$ to $t$  belonging to the (possibly degenerate) loop coded by $ \pi(s \wedge t)$ (this term is equal to $0$ if $ \pi(s \wedge t)$ is not a branch-point), see  \cref{fig:explan}.

 In particular, if $s \preccurlyeq t$, note that
  \begin{eqnarray} \label{eq:def1}{d}(s,t) &=&\delta_{s}({ 0},x_{s }^ {t})+ {d}_{0}(s,t)\quad =\quad    \sum_{s \preccurlyeq r \preccurlyeq t} \delta_{{r}}( 0, {x_{r }^ {t}}) .  \end{eqnarray}

\begin {lemma} [Bounds on $d$]\label{lem:controls} Let $r,s, t \in [0,1]$. Then:
\begin{enumerate}[(i)]
\item (Lower bound) If $s \prec r \prec t$,  we have $d(s,t) \geq  {\min}(x_{r}^{t} , \Delta_{r}-x_{r}^{t})$.
\item (Upper bound) If $s<t$,  we have $d(s,t) \leq  \X_{s}+\X_{t-}- 2 I_{s}^t$.
\end{enumerate}
\end {lemma} 

\proof
The first assertion is obvious from the definition of $d$ :
$$d(s,t) \geq  \delta_{r}(0,x_{r}^ {t}) \geq {\min}( x_{r}^{t} , \Delta_{r}-x_{r}^{t})$$
For $(ii)$,  let us first prove that if $s \prec t$ then \begin{equation}
\label{eq:ut}d_{0}(s,t) \leq \X_{t-}-I_{s}^t.
\end{equation} (Note that $ \X_{t-}-I_{s}^t \geq 0$ because $s \neq t$.) To this end, remark that if $s \preccurlyeq r \preccurlyeq t$ and $ s \preccurlyeq r' \preccurlyeq t$, then $ r \preccurlyeq r'$ or $ r' \preccurlyeq r$. It follows that 
 if  $ s  \prec r_0 \prec r_1 \prec \cdots \prec r_n = t$, using the fact that $I_{r_{{i}}}^{r_{{n}}}=I_{r_{{i}}}^{r_{{i+1}}}$ for $0 \leq i \leq n-1$, we have 
\begin{eqnarray*}
 \sum_{i=0}^  {n} \delta_{{r_i}}( 0,  {x_{ r_i }^ {r_n}}) & \leq &   \sum_{i=0}^  {n-1} {x_{ r_i }^ {r_n}} + { \delta_{r_{n}}(0,x_{r_{n}}^{r_{n}})} \\ 
  & = &  \sum_{i=0}^  {n-1} \left( I^  {r_{ {i+1}}}_ {r_ {i}} - \X_{r_ {i}-}\right) + 0 \\ & \leq&  \sum_{i=0}^  {n-1} \left( \X_ { r_ { {i+1}}-} - \X_ {r_ {i}-}\right)  = \X_ { r_ { {n}}-}- \X_ {r_ {0}-} 
  \leq \X_{t-}- I_{s}^t,
\end{eqnarray*}
where for the last inequality we have used the fact that  $I_{s}^{t} \leq   \X_{r_{0}-}$ since $s< r_{0} <t$.
Since  $d_{0}(s,t) = \sum_{s \prec r \preccurlyeq t} \delta_{r}(0,x_{r}^{t})$, this gives \eqref{eq:ut}. 

Let us return to the proof of $(ii)$. Let $s <t$. If $ s \prec t$, then by \eqref{eq:def1} and treating the jump at $s$ separately we can use \eqref{eq:ut} to get \begin{eqnarray*} d(s,t)&=&\delta_{s}(0,x_{s }^ {t})+d_{0}(s,t) \\ 
&\leq& { (\Delta_{s}-x_{s}^{t}) +  (\X_{t-}-I_{s}^t)} \\
&=& ( \X_{s-}+ \Delta_{s}- I_{s}^{t})+(\X_{t-}-I_{s}^t)= \X_{s}+\X_{t-}- 2 I_{s}^t.  \end{eqnarray*}
Otherwise $s \wedge t < s$. It is then easy to check that  
$I_{s}^t = I_{s \wedge t}^t$. In addition, $ \delta_{s\wedge t}( x_{s \wedge t}^{t}, x_{s \wedge t}^{s}) \leq {x_{s\wedge t }^ {s}}-{x_{s\wedge t }^ {t}}= I_{s \wedge t}^s-I_{s \wedge t}^t = I_{s \wedge t}^s - I_{s}^t$. Then by \eqref{eq:def2} and \eqref{eq:ut} we have
$$d(s,t)  \leq  I_{s \wedge t}^s - I_{s}^t+( \X_{t-}-I_{s\wedge t}^t)+ ( \X_{s-}-I_{s \wedge t}^s) = \X_{{s-}}+\X_{t-}- 2 I_{s}^t.$$
This completes the proof. \endproof

\begin {proposition} \label{prop:pseudodistance}Almost surely, the function $ d( \cdot, \cdot): [0,1] \times [0,1]  \rightarrow \R_+$ is a continuous pseudo-distance.
\end {proposition}

\begin{proof}By definition of $d$ and  \cref{lem:controls},  for every $s,t \in  [0,1]$, we have $d(s,t) \leq 2  \sup X^{ \mathrm{exc}} < \infty$. The fact that $d$ satisfies the triangular inequality is a straightforward but cumbersome consequence of its definition \eqref{eq:def2}. We leave the details to the reader. 

Let us now show that  the function $ d( \cdot, \cdot): [0,1]  \times [0,1]  \rightarrow \R_+$ is continuous. To this end, fix $ (s,t) \in [0,1]^2$ and let $ s_{n},t_{n}  (n \geq 1)$ be real numbers in $[0,1]$ such that $(s_{n}, t_{n}) \rightarrow (s,t)$ as $ n \rightarrow \infty$. The triangular inequality entails
$$ \left| d(s,t)- d(s_{n},t_{n}) \right| \leq d(s,s_{n})+d(t,t_{n}).$$
By symmetry, it is sufficient to show that $d(s,s_{n}) \rightarrow 0$ as $n \rightarrow \infty$. Suppose for a moment that $s_{n} \uparrow s$ and $s_{n}<s$, then by \cref{lem:controls}  $(ii)$ we have 
$$d (s_{n},s) \leq \X_{s_{n}}+ \X_{s-}-2 I_{s_{n}}^{s}  \quad\mathop{\longrightarrow}_{n \rightarrow \infty} \quad  \X_{s-}+\X_{s-}-2\X_{s-} =0.$$
The other case when  $s_{n} \downarrow s$ and $s_{n}<s$ is treated similarly. This proves the proposition.\end {proof}

We are finally  ready to define the looptree coded by $X^ \mathrm{exc}$. \begin{definition} \label{def:looptree} For $ x,y \in [0,1]$, set $ x\sim y$ if $ {d}(x,y)=0$. 
 The random stable looptree of index $\alpha$ is defined as the quotient metric space 
$$ \mathscr{L}_\alpha \quad = \quad \big([0,1]/\sim, {d} \big).$$
  \end{definition}
We will denote by $ \mathbf{p}$ the canonical projection $ \mathbf{p}: [0,1] \rightarrow \La$.   Since $d : [0,1]  \times [0,1] \to \mathbb{R}_{+}$ is a.s.\,continuous by \cref{prop:pseudodistance}, it immediately follows that $ \mathbf{p}: [0,1] \rightarrow \mathscr {L}_{ \alpha}$ is a.s.\,continuous. The metric space $ \mathscr{L}_\alpha$ is thus a.s.\,compact, as the image of a compact metric space by an a.s.\,continuous map. 

\medskip

With this definition, it is maybe not clear why $ \La$ contains loops. For sake of clarity, let us give an explicit description of these. Fix $s \in [0,1]$ with $\Delta_{s} >0$, and for $u \in [0,\Delta_{s}]$ let $s_{u} = \inf\{ t \geq s : X^{ \mathrm{exc}}_{t} = X_{s}^{ \mathrm{exc}}- u\}$. It is easy to check that the image of $\{s_{u}\}_{u \in [0, \Delta_{s}]}$ by $ \mathbf{p}$ in $ \mathscr{L}_{\alpha}$ is isometric to a circle of  length $ \Delta_{s}$, which corresponds to the loop attached to the branch-point  $ \pi(s)$ in the tree $ \Ta$.

To conclude this section, let us mention that it is possible to construct $ \mathscr{L}_{\alpha}$ directly from the stable tree $ \mathcal{T}_{ \alpha} $ in a measurable fashion. For instance, if $u= \pi(s)$, one can recover the jump $ \Delta_{s}$ as follows (see \cite[Eq.~(1)]{Mie05}):
 \begin{eqnarray}\label{eq:width}  \Delta_{s}  & \overset{a.s.}{=}&  \lim_{ \varepsilon \rightarrow 0}  \frac{1}{ \varepsilon} \mathsf{Mass}  \left\{  v \in  \mathcal{T}_{ \alpha}; d_{ \mathcal{T}_{\alpha} }(u,v) < \varepsilon\right\}, \end{eqnarray} where $ \mathsf{Mass}$ is the push-forward of the Lebesgue measure on $[0,1]$ by the projection $ \pi : [0,1] \rightarrow  \mathcal{T}_{ \alpha}$. 
However, we believe that our definition of $ \mathscr{L}_{\alpha}$ using LŽvy processes is simpler and more amenable to computations (recall also that the stable tree is itself defined by the height process $H^{ \mathrm{exc}}$ associated with $X^{ \mathrm{exc}}$).

\section{Properties of  stable looptrees}
\label{sec:properties}

The goal of this section is to prove  \cref{thm:dimension,thm:1and2}. Before doing so, we introduce some more background on spectrally positive stable LŽvy processes. This will be our toolbox for studying fine properties of looptrees. The interested reader should consult \cite{Ber92,Ber96,Cha97} for additional details. 

Let us stress that, to our knowledge, the limiting behavior of the normalized excursion of $ \alpha$-stable spectrally positive LŽvy processes  as $ \alpha \downarrow 1$ (\cref{prop:cvsautexc}) seems to be new.

\subsection{More on stable processes} 

\subsubsection{Excursions above the infimum} \label{sec:excmin}
In \cref{sec:normexc}, the normalized excursion process $X^ \mathrm{exc}$ has been introduced as the normalized excursion of $X$ above its infimum straddling  time $1$. Let us present another definition $X^ \mathrm{exc}$ using It™'s excursion theory (we refer to \cite[Chapter IV]{Ber96} for details). 

If $X$ is an  $\alpha$-stable spectrally positive LŽvy process, denote by $\underline{X}_{t}= \inf\{ X_{s} : 0 \leq s \leq t\}$ its running infimum process. Note that  $\underline{X}$ is continuous since $X$ has no negative jumps. The process $X - \underline{X}$ is  strong Markov  and $0$ is regular for itself, allowing the use of excursion theory. We may and will choose $- \underline{X}$ as the
local time of $X- \underline{X}$ at level $0$.  Let $(g_j,d_j), j \in \mathcal{I}$
be the excursion intervals of $X- \underline{X}$ away from $0$. For every $j \in \mathcal{I}$ and $ s \geq  0$, set $\omega_s^j= X_{(g_j+s) \wedge d_j}-X_{g_j}$.
We view $ \omega^j$ as an element of the excursion space $ \mathcal {E}$,  defined by:
$$ \mathcal{E}=  \{ \omega \in \D(\R_+, \R_+); \, \omega(0)=0 \textrm{ and } \z( \omega):=
 \sup\{s>0; \omega(s)>0\,\} \in (0, \infty)  \}.$$
If $\omega\in \mathcal{E}$, we call $ \zeta( \omega)$ the lifetime of the excursion $ \omega$. From It™'s excursion theory, the point measure
$$\mathcal{N}( dt d \omega)= \sum_{j \in \mathcal{I}} \delta_{(-\underline{X}_{g_j},\omega^j)}$$
is a Poisson measure with intensity $dt \underline{n}(d\omega)$, where $\underline{n}(d
\omega)$ is a $\sigma$-finite measure on the set $ \mathcal{E}$ called the It™ excursion measure. This measure admits the following scaling property. For every $\lambda>0$, define $S^{(\lambda)}: \mathcal{E} \rightarrow \mathcal{E}$
 by $S^{(\lambda)}(\omega)=\left( \lambda^{1/ \alpha} \omega(s/\lambda), \,  s \geq 0 \right)$. Then (see \cite{Cha97}
or \cite[Chapter VIII.4]{Ber96} for details) there exists a unique collection of probability measures
$(\underline{n}_{(a)}, a>0)$ on the set of excursions such that the
following properties hold:
\begin{enumerate}[(i)]
\item[$(i)$] For every $a>0$, $\underline{n}_{(a)}(\z=a)=1$.
\item[$(ii)$] For every $\lambda>0$ and $a>0$, we have
$S^{(\lambda)}(\underline{n}_{(a)})=\underline{n}_{(\lambda a)}$.
\item[$(iii)$] For every measurable subset $A$ of the set of all excursions: $$ \underline{n}(A)= \int_0 ^\infty \underline{n}_{(a)}(A) \frac{da}{\alpha \Gamma(1-1/\alpha)
a^{1/\alpha+1}}.$$
\end{enumerate}
In addition, the probability distribution $\underline{n}_{(1)}$, which is supported on the cˆdlˆg paths with unit
lifetime, coincides with the law of $ \X$ as defined in \cref{sec:normexc}, and is also denoted by $ \underline{n}( \cdot | \zeta=1)$.  Thus, informally, $ \underline{n}( \cdot | \zeta=1)$ is the law of an
excursion under the It™ measure conditioned to have unit lifetime.

 \subsubsection{Absolute continuity relation for $X^ \mathrm{exc}$} \label{sec:absolutecontinuity}
 
We will use a path transformation due to Chaumont \cite{Cha97} relating the bridge of a stable LŽvy process to its normalized excursion, which generalizes the Vervaat transformation in the Brownian case. If $U$ is a uniform variable over $[0,1]$ independent of $X^{ \mathrm{exc}}$, then the process {$X^{ \mathrm{br}}$ defined by}$$ X^{ \mathrm{br}}_{t} =  \left\{ \begin{array}{ll} X^{ \mathrm{exc}}_{U+t} &\mbox{ if } U+t \leq 1,\\ X^{ \mathrm{exc} 	}_{t+U-1} & \mbox{ if } U +t  >1. \end{array}  \right.\quad \mbox{for } t \in [0,1]$$ is distributed {according to} the bridge of the stable process $X$, which
can informally be seen as the {process} $(X_t; \, 0 \leq t \leq 1)$
conditioned to be at level zero at time one. See \cite[Chapter
VIII]{Ber96} for definitions. In the other direction, to get $ X^{ \mathrm{exc}}$ from $X^{ \mathrm{br}}$ we just re-root $ X^{ \mathrm{br}}$ {by performing a cyclic shift} at the (a.s.\,unique) time $u_{ \star}( X^ \br)$ where it attains its minimum. 
 
We finally state an absolute continuity property between $ X^ \br$ and $ \X$. Fix $a \in (0,1)$. Let $F: \D([0,a],\R) \rightarrow \R $ be a bounded continuous function. We have (see  \cite[Chapter
VIII.3, Formula (8)]{Ber96}):
$$\Es{F\left( X^{\br}_t; \, 0 \leq t \leq a \right)}=\Es{F\left( X_t; \, 0 \leq t \leq a \right)
\frac{p_{1-a}(-X_a)}{p_1(0)}},$$
where $p_{t}$ is the density of $X_{t}$. 
Note that by  time reversal, the law of $(X_{1}-X_{(1-t)-})_{0 \leq t \leq 1}$ satisfies the same property. 

The previous two results will be used in order to reduce the proof of a statement concerning $ \X$ to a similar statement involving $X$ (which is usually easier to obtain). More precisely, a property concerning $X$ will be first transferred to $ X^ \br$ by absolute continuity, and then to $ \X$ by using the Vervaat transformation.

\subsubsection{Descents}
\label{sec:descents}

 Let $Y:  { \R} \rightarrow \R$ be cˆdlˆg function. For every   { $s,t \in \R$}, we write $s \preccurlyeq_{Y} t$ if and only if $s \leq t$ and $Y_{s-} \leq \inf_{[s,t]} Y$, and in this case we set $$ x_{s}^{t}(Y) = \inf_{[s,t]} Y- Y_{s-} \geq 0, \quad \mbox{and} \quad u_{s}^{t}(Y) = \frac{x_{s}^{t}(Y)}{ \Delta Y_{s}} \in [0,1].$$ 
 We write $s \prec_{Y} $ if $s \preccurlyeq_{Y} t$ and $s \ne t$. When there is no ambiguity, we write $ x_{s}^t$ instead of $x_{s}^t(Y)$, etc. For $t   { \in \R}$, the collection $\{(x_{s}^{t}(Y), u_{s}^{t}(Y)) : s  \preccurlyeq t\}$  is called the \emph{descent} of $t$ in $Y$.  As the reader may have noticed, this concept is crucial in the definition of the distance involved in the definition of stable looptrees. 
 
We will describe the law of the descents (from a typical point) in an $\alpha$-stable LŽvy process by using excursion theory. To this end, denote $\overline{X}_{t}= \sup\{ X_{s} : 0 \leq s \leq t\}$ the running supremum process of $X$.  The process $\overline{X}-X$ is  strong Markov and $0$ is regular for itself. Let $({L}_{t},t \geq 0)$ denote a local time of $\overline{X}-X$ at level $0$, normalized in such a way that $ \Es{ \exp(- \lambda \overline{X}_{L^{-1}(t)})}= \exp(- t \lambda^{ \alpha-1})$. Note that by  \cite[Chapter VIII, Lemma 1]{Ber96}, ${L}^{-1}$ is a stable subordinator of index $1-1/ \alpha$. Finally, to simplify notation, set $ \mathrm{x}_{s}=\overline{X}_{s}-\overline{X}_{s-}$ and $ \mathrm{u}_{s}= \frac{\overline{X}_{s}-\overline{X}_{s-}}{\overline{X}_{s}-{X}_{s-}}$ for every $s \geq 0$ such that $\overline{X}_{s}> \overline{X}_{s-}$. {In order to describe the law of descents from a fixed point in an $\alpha$-stable process we need to introduce the two-sided stable process. If $X^1$ and $X^2$ are two independent stable processes on $ \R_{+}$, set $X_{t}=X^1_{t}$ for $ t \geq 0$ and $X_{t}=-X^2_{(-t)-}$ for $t<0$.}

\begin{proposition}  \label{prop:itomeasure}The following assertions hold.
\begin{enumerate}
 \item[$(i)$] Let $(X_{t} : t \in \mathbb{R})$ be a two-sided spectrally positive $\alpha$-stable process. Then the collection
$$\{(-s,x_{s}^0(X),u_{s}^0(X)) :  s \preccurlyeq 0\}$$ has the same distribution as
$$ \left\{(s, \mathrm{x}_{s}, \mathrm{u}_{s}) ; s \geq 0 \textrm{ s.t. }\overline{X}_{s}> \overline{X}_{s-} \right\}.$$ 
 \item[$(ii)$] The point measure
\begin{equation}
\label{eq:poisson}\sum_{ \overline{X}_{s}> \overline{X}_{s-}} \delta_{(L_{s},  \frac{\mathrm{x}_{s}}{ \mathrm{u}_{s}},  \mathrm{u}_{s})}
\end{equation}
is a Poisson point measure with intensity $dl \cdot x\Pi(dx) \cdot \mathbbm{1}_{[0,1]}(r)dr$. 
  \end{enumerate}

 \end{proposition}
 \proof The first assertion follows from the fact that the dual process $ \hat{X}$, defined by $ \hat{X}_{s} = -X_{(-s)-}$ for $s \geq 0$, has the same distribution as $X$ and that 
 $$(x_{-s}^0(X),u_{-s}^0(X))=\left(\overline{ \hat{X}}_{s}-\overline{ \hat{X}}_{s-}, \frac{\overline{ \hat{X}}_{s}-\overline{ \hat{X}}_{s-}}{\overline{ \hat{X}}_{s}-{ \hat{X}}_{s-}} \right)$$
 for every $s \geq 0$ such that $-s \preccurlyeq 0$, or equivalently $\overline{ \hat{X}}_{s}> \overline{ \hat{X}}_{s-}$.
 
 For $(ii)$, denote by $(g_{j},d_{j})_{j \in J}$ the excursion intervals of $\overline{X}-X$ above $0$. It is known (see  { \cite[Corollary 1]{Ber92}}) that the point measure
$$ \sum_{j \in J} (L_{g_{j}}, \Delta X_{d_{j}},\Delta \overline{X}_{d_{j}})$$
is a Poisson point measure with intensity $dl \cdot \Pi(dx) \cdot \mathbbm{1}_{[0,x]}(r)dr$. The conclusion follows.
  \endproof 

  We now state a technical but useful consequence of the previous proposition, which will be required in the proof of the lower bound of the Hausdorff dimension of stable looptrees.

 \begin{corollary} \label{cor:technique}Fix $ \eta>0$. Let $(X_{t} : t \in \mathbb{R})$ be a two-sided $\alpha$-stable process. For $ \epsilon>0$, set
 $$A_{ \varepsilon} = \left\{ \exists s \in [-\varepsilon,0] \mbox{ with } s \preccurlyeq 0 : \begin{array}{c} x_{s}^{0}(X) \geq  \varepsilon^{1/\alpha +\eta}\\ \mbox{ and } \\ \Delta X_{s} - x_{s}^{0}(X) \geq \varepsilon^{1/\alpha +\eta} \end{array}\right\}.$$
 Then $ \P(A_{ \varepsilon}^c) \leq C \epsilon^{ \gamma}$ for certain constants $ C, \gamma>0$ (depending on $ \alpha$ and $ \eta$).
 \end{corollary}
 
\begin{proof} Set  {$B_{ \varepsilon} = \{ \exists s \in [0,\varepsilon] :  X_{s} \geq  \varepsilon^{1/\alpha +\eta} \mbox{ and } \Delta X_{s} - x_{s}\geq \varepsilon^{1/\alpha +\eta}  \}$}. By \cref{prop:itomeasure} $(i)$, it is sufficient to establish the existence of two constants $C, \gamma>0$ such that $ \P(B_{\varepsilon}^c) \leq C \epsilon^{ \gamma}$. To simplify notation, set $\overline{ \alpha}=1-1/ \alpha$ and $ c_{ \epsilon}= \epsilon^ { \eta ( \alpha-1)/2}$. Then write:
\begin{eqnarray}
\P(B_{\varepsilon}^c) &\leq& \P(B_{\varepsilon}^c, L_{ \varepsilon} >  c_{ \epsilon} \epsilon^{\overline{ \alpha}}) + P(L_{ \varepsilon} < c_{ \epsilon} \epsilon^{\overline{ \alpha}})  \notag\\
& \leq & {\P( \forall s \textrm{ s.t. } L_{s} \leq  c_{ \epsilon} \epsilon^{\overline{ \alpha}} :  x_{s} <  \varepsilon^{1/\alpha +\eta} \mbox{ or } \Delta X_{s} - x_{s}< \varepsilon^{1/\alpha +\eta})}+ \P(L_{ \varepsilon} < c_{ \epsilon} \epsilon^{\overline{ \alpha}}). \label {eq:r}
\end{eqnarray}
{Using the fact that  \eqref{eq:poisson} is a Poisson point measure with intensity $dl \cdot x\Pi(dx) \cdot \mathbbm{1}_{[0,1]}(r)dr$, it follows that the first term of  \eqref{eq:r} is equal to
\begin{eqnarray*}
\exp \left( - c_{ \epsilon} \epsilon^{\overline{ \alpha}}  \int_{0}^1 dr  \int_{ - \infty} ^ { \infty} x \Pi(dx)  \, \mathbbm {1}_{  \{ rx \geq  \varepsilon^{1/\alpha +\eta} \mbox { and }  x (1-r) \geq \varepsilon^{1/\alpha +\eta}\}}\right) &=& \exp \left( - c_{ \epsilon} \epsilon^{\overline{ \alpha}}  \cdot  c \epsilon^{- \eta( \alpha-1)}\right)
\end{eqnarray*}
for a certain constant $c>0$. In addition, 
$$\P(L_{ \varepsilon} < c_{ \epsilon} \epsilon^{\overline{ \alpha}}) \leq  \P \left(L^{-1}_{1}> \frac{\varepsilon}{( c_{ \epsilon} \epsilon^{\overline{ \alpha}})^ {1/ \overline{ \alpha}}} \right) \leq \P(L^{-1}_{1}>  {\epsilon}^{-  \eta \alpha/2}).$$
The conclusion follows since $\P(L^{-1}_{1}>u) = \mathcal{O}(u^{ -\bar{\alpha}})$ as $u \rightarrow \infty$}.\end{proof}

 We conclude this section by a  lemma which will be useful in the proof of \cref{thm:continuity}. See also \cite[Proof of Proposition 7]{LGM09} for a similar statement.

\begin{lemma}\label{lem:jumps}Almost surely, for every $ t \geq 0$ we have 
 \begin{eqnarray} \label{eq:jump} X_{t} - \inf_{[0,t]} X  &=& \sum_{\begin{subarray}{c}s \preccurlyeq t \\ s \geq 0 \end{subarray}} x_{s}^{t}(X).  \end{eqnarray}
\end{lemma}

\begin{proof} The left-hand side of the equality appearing in the statement of the lemma is clearly a cˆdlˆg function. It also simple, but tedious, to check that the right-hand side is a cˆdlˆg function as well. It thus suffices to prove that \eqref{eq:jump} holds almost surely for every fixed $t \geq 0$. 

Set $\hat{X}_{s}=X_{(t-s)-}-X_{t-}$ for $ 0 \leq s \leq t$, and to simplify notation set $ S_{u}= \sup_{[0,u]} \hat{X}$. In particular, $(X_{s},0 \leq s \leq t)$ and $ (\hat{X}_{s}, 0 \leq s \leq t)$ have the same distribution. Hence
\begin{equation}
\label{eq:lemj} \big(S_{t}, \sum_{  0 \leq  s \leq  t}  \Delta  S_{s}  \big) \quad \mathop{=}^{(d)}  \quad \Big(X_{t} - \inf_{[0,t]} X, \sum_{\begin{subarray}{c}s \preccurlyeq t \\ s \geq 0 \end{subarray}} x_{s}^{t}(X) \Big).
\end{equation}
Then notice that ladder height process $(S_{ L^{-1}_{t}}, t \geq 0)$ is a subordinator without drift  
\cite[Chapter VIII, Lemma 1]{Ber96}, hence a pure jump-process. 
This implies that $S_{t}$ is the sum of its jumps, i.e. a.s $S_{t}= \sum_{0 \leq s \leq t} \Delta S _{s}$. This completes the proof of the lemma.
\end{proof}

The following result is the {analog statement} {for} the normalized excursion.
 
\begin{corollary} \label{coro:exc} {Almost surely,} for every $t \in [0,1]$ we have 
$$ X_{t}^{ \mathrm{exc}} = \sum_{0 \preccurlyeq s \preccurlyeq t} x_{s}^t ( \X).$$
\end{corollary}
\begin {proof} This follows from the previous lemma and the construction of $X^{ \mathrm{exc}}$ as the normalized excursion above the infimum of $X$ straddling time $1$ {in Section \ref{sec:normexc}}. We leave details to the reader.
\end {proof}

In particular \cref{coro:exc} implies that {almost surely}, for every $0 \leq t \leq 1$,
\begin{equation}
\label{eq:u}X_{t}^{ \mathrm{exc}} =\displaystyle\sum_{0 \preccurlyeq s \preccurlyeq t} \Delta X^{\exc}_{s} \cdot u_{s}^t( X^{\exc}). \end{equation}
By \eqref{eq:def1}, a similar equality, which will be useful later,  holds {almost surely for every $ 0 \leq t \leq 1$}:
\begin{equation}
\label{eq:u2}d(0,t)=\displaystyle \sum_{0 \preccurlyeq s \preccurlyeq t} \Delta X^{\exc}_{s} \cdot  \min\big( u_{s}^t( X^{\exc}),1-u_{s}^t( X^{\exc})\big).\end{equation}

\subsubsection{Limiting behavior of the normalized excursion as $ \alpha \downarrow 1$ and $ \alpha \uparrow 2$}

In this section we study the behavior of $X^{ \mathrm{exc}}$ as $\alpha \to 1$ or $ \alpha \to 2$. In order to stress the dependence in $ \alpha$, we add an additional superscript $^{(\alpha)}$, e.g. $X^{(\alpha)},X^{ \mathrm{br},(\alpha)},X^{\exc,(\alpha)}$  will respectively denote the $ \alpha$-stable spectrally positive process, its bridge and normalized excursion, and $ \Pi^{( \alpha)}, \underline{n}^{(\alpha)}$ will respectively denote the LŽvy measure and the excursion measure above the infimum of $X^{(\alpha)}$.

\paragraph{Limiting case $ \alpha \uparrow 2$.}

We prove that $X^{\exc,(\alpha)}$ converges, as $ \alpha \uparrow 2$, towards a multiple of the normalized Brownian excursion, denoted by $\mathbf{e}$ (see \cref{fig:levys} for an illustration). This is standard and should not be surprising, since the $ \alpha=2$ stable LŽvy process is just $\sqrt{2}$ times Brownian motion.

\begin{proposition} \label{prop:cve}The following convergence holds in distribution for the topology of uniform convergence on every compact subset of $ \mathbb{R}_{+}$
\begin{eqnarray} \label{eq:case2}  X^{\exc,(\alpha)} \quad \xrightarrow[\alpha\uparrow 2]{(d)} \quad \sqrt{2} \cdot \mathbf{e}.  
\end{eqnarray}
\end {proposition}

\begin {proof}
We first establish an unconditioned version of this convergence. Specifically, if $B$ is a standard Brownian motion, we show that
\begin{equation}
\label{eq:cvB} X^{(\alpha)}  \quad\mathop{\longrightarrow}^{(d)}_{ \alpha \uparrow 2} \quad  \sqrt{2} \cdot B, 
\end{equation} where the convergence holds in distribution for the uniform topology on $ \D([0,1], \R)$. Since $B$ is almost surely continuous, by \cite[Theorems V.19, V.23]{Pol84} it is sufficient to check that the following three conditions hold as $ \alpha \uparrow 2$: 
\begin{enumerate}
\item[(a)] {The convergence} $  \displaystyle X^{(\alpha)}_{0} \mathop{\longrightarrow}  \sqrt {2} \cdot B_{0}$ {holds} in distribution,
\item[(b)] For every $0<s<t$, the convergence $ \displaystyle X^{(\alpha)}_{t}-X^{(\alpha)}_{s}  \mathop{\longrightarrow}  \sqrt {2} \cdot (B_{t}-B_{s})$ {holds} in distribution,
\item[(c)] For every $ \delta>0$, there exist $ \eta, \epsilon>0$  such that for $0 \leq s \leq t \leq 1$:
 $$ | t-s | <  \eta \qquad \Longrightarrow \qquad \Pr{|X^{(\alpha)}_{t}-X^{(\alpha)}_{s}  | \leq \delta/2} \geq \epsilon.$$
  \end{enumerate}
It is clear that Condition (a) holds. The scaling property of $ X^{( \alpha)}$ entails that $X^{(\alpha)}_{t}-X^{(\alpha)}_{s}$ has the same law as $(t-s)^{1/ \alpha} \cdot X^{( \alpha)}_{1}$. On the other hand, for every $u \in \mathbb{R}$, we have 
$$ \Es{\exp( \mathrm{i}u X^{(\alpha)}_{1})} \quad  \xrightarrow[ \alpha \uparrow 2]{} \quad \exp(-u^2) = \Es{\exp(\mathrm{i} u \sqrt{2} B_{1})}.$$ Condition (b) thus holds. The same argument gives Condition (c). 
This establishes \eqref{eq:cvB}.

 The convergence \eqref{eq:case2} is then a consequence of the construction of $X^{\exc,(\alpha)}$ from the excursion of $X^{(\alpha)}$ above its infimum straddling time $1$ (see \cref{sec:normexc}). Indeed, by Skorokhod's representation theorem, we may assume that the convergence \eqref{eq:cvB} holds almost surely. Then set $$ \underline{g}_{1}^{(\alpha)} =  \sup\{s \leq 1 : X^{(\alpha)}_{s} = \inf_{[0,s]} X^{(\alpha)}\} \quad  \mbox{and} \quad  \underline{d}_{1}^{(\alpha)} =  \inf\{s > 1 : X^{(\alpha)}_{s} = \inf_{[0,s]} X^{(\alpha)}\}.$$ Similarly, define $ \underline{g}_{1}^{(2)},\underline{d}_{1}^{(2)}$  when $X^{(\alpha)}$ is replaced by  $ \sqrt{2} \cdot B$.  Since local minima of Brownian motion are almost surely distinct, we get that $ \underline{g}_{1}^{(\alpha)} \to \underline{g}_{1}^{(2)}$ a.s.\,as $ \alpha \uparrow 2$. On the other side, since for every $ \alpha \in (1,2]$, a.s.\,$ \underline{d}_{1}^{(2)}$ is not a local minimum of $B$ (this follows from the Markov property applied at the stopping time $ \underline{d}_{1}^{(2)})$  we get that $ \underline{d}_{1}^{(\alpha)} \to \underline{d}_{1}^{(2)}$ in distribution as $ \alpha \uparrow 2$. The desired convergence \eqref{eq:case2} then follows from \eqref{eq:Xexc}.
\end {proof}

\paragraph{Limiting case $ \alpha \downarrow 1$.}
The limiting behavior of the normalized excursion $X^ {\mathrm{exc},(\alpha)}$ as $\alpha \downarrow 1$ is very different from the case $ \alpha \uparrow 2$. Informally, we will see that in this case,  $X^ {\mathrm{exc},(\alpha)}$  converges towards  the deterministic affine function on $[0,1]$ which is equal to $1$ at time $0$ and $0$ at time $1$. Some care is needed in the formulation of this statement, since the function $x \mapsto \mathbbm{1}_{0<x\leq1}(1-x)$ is not cˆdlˆg. To cope up with this technical issue, we reverse time:

{\begin {proposition} \label{prop:cvsautexc}The following convergence holds in distribution in $ \D([ 0,1] , \R)$:
 $$  \left( X^ {\mathrm{exc},(\alpha)}_ { (1-t)-}, 0 \leq t \leq 1 \right) \quad \mathop { \longrightarrow}^ {(d)}_ {\alpha \downarrow 1} \quad \left( t \mathbbm {1}_{  \{t \neq 1\}}, 0 \leq t \leq 1 \right).$$
\end {proposition}
}

\begin{remark}Let us mention here that the case $\alpha\downarrow1$ is not (directly) related to Neveu's branching process \cite {Nev92} which is often considered as the limit of a stable branching process when $\alpha \to 1$. Indeed, contrary to the latter, the limit of $ X^{ \mathrm{exc}, \alpha}$ when $\alpha  \downarrow 1$ is deterministic. The reason is that Neveu's branching process has LŽvy measure $r^{-2} \mathbbm {1}_{(0, \infty)}dr$, but recalling our normalization \eqref{eq:levymeasure}, in the limit  $\alpha \downarrow 1$, the LŽvy measure $\Pi^{(\alpha)}$ does \emph{not} converge to $r^{-2} \mathbbm {1}_{(0, \infty)}dr$.
\end{remark}

\cref{prop:cvsautexc} is thus a new ``one-big jump principle'' (see \cref{fig:levys} for an illustration), which is a well-known phenomenon in the context of subexponential distributions (see \cite {FKZ13} and references therein). {See also \cite{AL11,D80} for similar one-big jump principles.}

  \begin{figure}[!h]
 \begin{center}
 \includegraphics[width= 0.70 \linewidth]{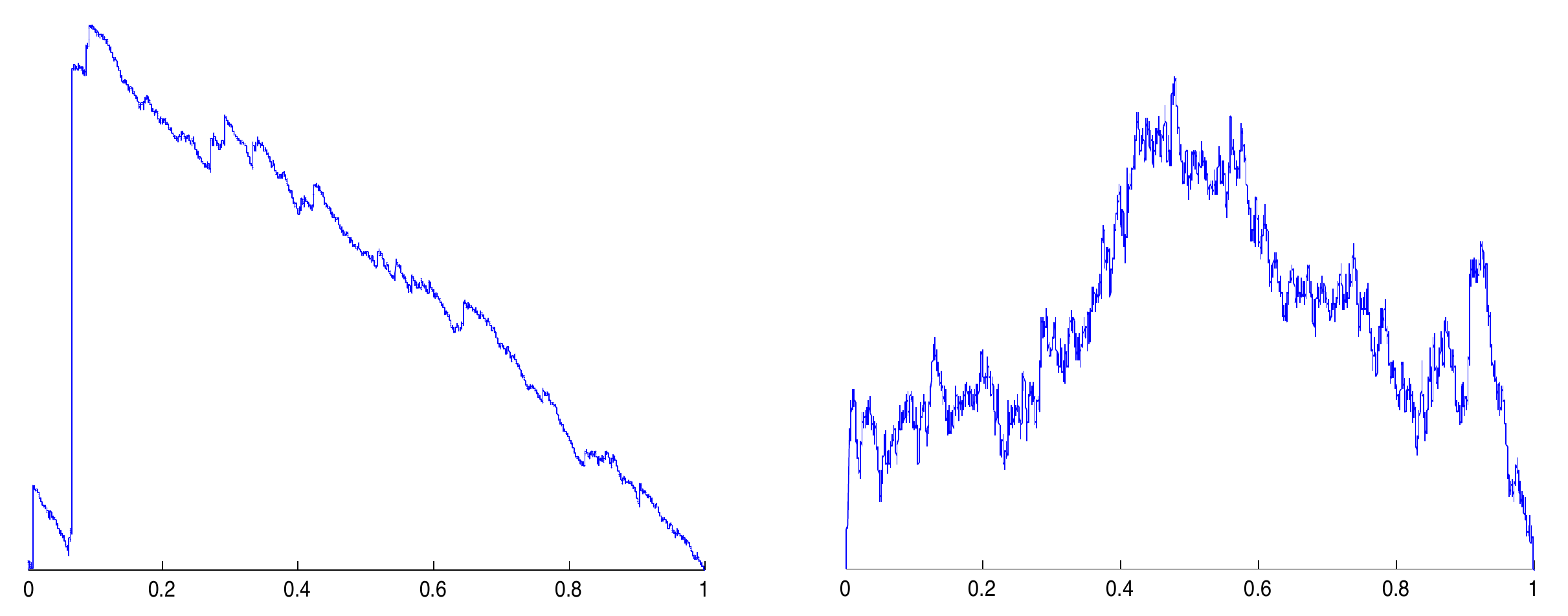}
 \caption{Simulations of $X^ {\mathrm{exc},(\alpha)}$ for respectively $ \alpha=1.00001$ and $ \alpha=1.99$.}
 \label{fig:levys}
 \end{center}
 \end{figure}
 
  The strategy to prove \cref{prop:cvsautexc} is first to establish the convergence of $X^{\exc,(\alpha)} $ on every fixed interval of the form $[ \varepsilon,1]$ with $ \epsilon \in (0,1)$ and then to {study} the behavior {near} $0$.
  
\begin{lemma} \label{cv:Xeps}For every $ \varepsilon \in (0,1)$,  \begin{eqnarray*} \big(X^{ \mathrm{exc}, (\alpha)}_{t}; \varepsilon \leq t \leq 1\big)& \xrightarrow[\alpha \downarrow 1]{( \P)} & (1-t; \, \varepsilon \leq t \leq 1),  \end{eqnarray*} 
where the convergence holds in probability for {the} uniform norm.
\end{lemma}

\proof[Proof of  \cref{cv:Xeps}]  Following the spirit of the proof of  \cref{prop:cve}, we first establish an analog statement for the unconditioned process $X^{(\alpha)}$ by proving that
\begin{equation}
\label{eq:cvdrift} X^{(\alpha)}  \quad\mathop{\longrightarrow}^{(d)}_{ \alpha \downarrow 1} \quad  ( -t ; \, t \geq 0), 
\end{equation} where the convergence holds in distribution for the uniform convergence on every compact subset of $ \mathbb{R}_{+}$. To {establish} \eqref{eq:cvdrift}, we also rely on \cite[Theorems V.19, V.23]{Pol84} and easily check that Conditions $(a),(b)$ and $(c)$  hold, {giving} \eqref{eq:cvdrift}. Fix $ \varepsilon  \in (0, {1}/{10})$. We shall use the notation $[a \pm b] := [a-b, a+b]$ for $a \in \mathbb{R}$ and $b >0$. We also introduce the functions $ \ell(s) = 1-s$  and $ \ell_{ \varepsilon}(s) = 1- \varepsilon-s$ for $s \in [0,1]$. To prove the lemma, we show that for every $ \varepsilon >0$ we have 
$$ \underline{n}^{(\alpha)}\big(  \{ \omega_{t} \in [1-t \pm  \varepsilon], \ \forall t \in [ \varepsilon, 1]\} \mid \zeta =1\big) \quad \xrightarrow[\alpha \downarrow 1]{} \quad 1.$$
By the scaling property of the measure $\underline{n}^{(\alpha)}$ (see {property (iii) in} \cref{sec:excmin}), it is sufficient to show that 
  \begin{eqnarray} \label{eq:goalscal} \underline{n}^{(\alpha)}\Big(  \sup_{t \in [ \varepsilon, \zeta]}|\omega_{t} - \ell(t)| \leq 10 \varepsilon \mid \zeta \in [1 \pm  \varepsilon]\Big) & \xrightarrow[ \alpha \downarrow 1]{} & 1.  \end{eqnarray} 
  {For $t  \geq 0$, denote by  $q_{t}^{(\alpha)}(dx)$ the entrance measure at time $t$ under $ \underline{n}^{( \alpha)}$, defined by relation
 $$ \underline{n}^{( \alpha)} \left( f( \omega_{t}) \mathbbm{1}_{  \{ \zeta>t\}}\right)= \int_{0}^ \infty f(x) q^{(\alpha)}_{t}(dx)$$
  for every measurable function $f: \R_{+} \rightarrow \R_{+}$
  }
Then, {using the fact that, for every $t>0$, under the conditional probability measure $\underline{n}^{( \alpha)}( \, \cdot \, | \z>t)$, the process
$( \omega_{t+s})_{s \geq 0}$ is Markovian with entrance law $q^{(\alpha)}_{t}(dx)$ and transition kernels of $X^{( \alpha)}$
stopped upon hitting $0$, we get
}
 \begin{eqnarray} && \underline{n}^{(\alpha)}\Big(  \sup_{t \in [ \varepsilon, \zeta]}|\omega_{t} - \ell(t)| \leq 10 \varepsilon \mid \zeta \in [1 \pm  \varepsilon]\Big) \nonumber \\
 & = &  \frac{1}{\underline{n}^{(\alpha)}( \zeta \in [1 \pm \varepsilon])}\int_{0}^\infty q_{ \varepsilon}^{(\alpha)}(dx) P_{x}^{(\alpha)} \Big ( \sup_{t\in [0, \tau]}| X^{(\alpha)}_{t} -  \ell_{ \varepsilon}(t)| \leq 10 \varepsilon \ \mbox{ and } \ \tau\in[1-\varepsilon \pm  \varepsilon] \Big), \nonumber\\
  \label{eq:entrance}  \end{eqnarray} where $P_{x}^{(\alpha)}$ denotes the distribution of a standard $\alpha$-stable  process $X^{(\alpha)}$ started from $x$ and stopped at the first time $\tau$ when it touches $0$. From  \eqref{eq:cvdrift} it follows that for every $\delta \in (0, \varepsilon)$ the convergence
  \begin{eqnarray*} P_{x}^{(\alpha)} \Big ( \sup_{[0, \tau]} |X^{(\alpha)} - \ell_{ \varepsilon}| \leq 10 \varepsilon  \mbox{ and } \tau\in[1-\varepsilon \pm  \varepsilon] \Big) & \xrightarrow[\alpha\downarrow1]{}& 1  \end{eqnarray*}
  {holds uniformly in $x \in [1- \varepsilon \pm (\varepsilon- \delta)].$} Consequently
   \begin{eqnarray} \label{liminf} \liminf_{\alpha \downarrow 1} \frac{ \displaystyle \int_{0}^\infty q_{ \varepsilon}^{(\alpha)}(dx) P_{x}^{(\alpha)} \Big (  \sup_{[0, \tau]} ||X^{(\alpha)} - \ell_{ \varepsilon}||\leq 10 \varepsilon \mbox{ and } \ \tau\in[1-\varepsilon \pm  \varepsilon] \Big)}{ \displaystyle \int_{0}^\infty q_{ \varepsilon}^{(\alpha)}(dx)  \mathbbm{1}_{x \in [1- \varepsilon \pm ( \varepsilon-\delta)]}} \geq 1.  \end{eqnarray}
   On the other hand, we can write  provided that $2 \delta<  \varepsilon$ (notice that $ 1- \varepsilon + 2 \delta > \varepsilon$)
   $$ \underline{n}^{(\alpha)}( \zeta \in [1 \pm ( \varepsilon-2 \delta)]) = \int_{0}^\infty q_{ \varepsilon}^{(\alpha)}(dx) P_{x}^{(\alpha)} \Big (\tau\in[1-\varepsilon \pm ( \varepsilon- 2 \delta)] \Big).$$ 
Convergence \eqref{eq:cvdrift} then entails that $ g(x,\alpha) := P_{x}^{(\alpha)} (\tau\in[1-\varepsilon \pm ( \varepsilon-2 \delta)] )$ also tends towards $0$ as $\alpha\downarrow1$, uniformly for $x \in \mathbb{R}_{+} \backslash [1- \varepsilon \pm ( \varepsilon - \delta)]$.    Since the total mass $ \int_{0}^\infty q_{ \varepsilon}^{(\alpha)} (dx) = \underline{n}^{(\alpha)}( \zeta > \varepsilon)$ is finite, the dominated convergence theorem implies that $$\int_{\mathbb{R}_{+} \backslash [1- \varepsilon \pm ( \varepsilon - \delta)]} q_{ \varepsilon}^{ (\alpha)}(dx) g(x,\alpha)  \quad\xrightarrow[ \alpha \downarrow 1]{}  \quad 0.$$
Finally, as $g(x,\alpha)$ is bounded by $1$ we get by dominated convergence and the last display that
         \begin{eqnarray} \label{limsup}   \liminf_{\alpha \downarrow 1} \frac{\displaystyle \int_{0}^\infty q_{ \varepsilon}^{(\alpha)}(dx)  \mathbbm{1}_{x \in [1- \varepsilon \pm ( \varepsilon-\delta)]}}{\underline{n}^{(\alpha)}( \zeta \in [1 \pm ( \varepsilon-2 \delta)])}  = {\liminf_{\alpha \downarrow 1} \frac{\displaystyle \int_{0}^\infty q_{ \varepsilon}^{(\alpha)}(dx)  \mathbbm{1}_{x \in [1- \varepsilon \pm ( \varepsilon-\delta)]}}{\displaystyle \int_{0}^\infty q_{ \varepsilon}^{(\alpha)}(dx) g(x, \alpha) }} & \geq & 1.\end{eqnarray}
    Combining \eqref{liminf} and \eqref{limsup} with \eqref{eq:entrance} we deduce that 
    \begin{equation}
    \label{eq:e}\liminf_{\alpha \downarrow 1} \underline{n}^{(\alpha)}\Big(  \sup_{t \in [ \varepsilon, \zeta]} | \omega_{t} - \ell(t)| \leq 10 \varepsilon \mid \zeta \in [1 \pm  \varepsilon]\Big) \geq \frac{\underline{n}^{(\alpha)}( \zeta \in [1 \pm ( \varepsilon-2 \delta)])}{\underline{n}^{(\alpha)}( \zeta \in [1 \pm  \varepsilon])}.
    \end{equation} Since $\underline{n}^{(\alpha)}( \zeta >t) = t^{-1/\alpha} {/ \Gamma(1-1/ \alpha)}$ by property (iii) in \cref{sec:excmin}, it follows that  the right-hand side of \eqref{eq:e} tends to $1$ as $ \delta \rightarrow 0$. This completes the proof. \endproof 

We have seen in \cref{cv:Xeps} that $X^{\exc,(\alpha)}$ converges to the deterministic function $x \mapsto 1-x$ over  every interval $[ \varepsilon,1]$ for {every} $ \varepsilon>0$. Still, this {does not imply}  \cref{prop:cvsautexc} because, { as $\alpha \downarrow 1$, the difference of magnitude roughly $1$ between times $0$ and $ \epsilon$ could be caused by the accumulation of many small jumps of total sum of order $1$ and not by a single big jump of order $1$. We shall show that this is not the case by using the LŽvy bridge $X^{ \mathrm{br}, (\alpha)}$} and  and a shuffling argument.

\proof[Proof of \cref{prop:cvsautexc}.]
For $ \varepsilon>0$ and $Y : [0,1] \to \mathbb{R}$, {let $\mathsf{J}(Y, \varepsilon)$}  be the set defined by
$$ \mathsf{J}(Y, \varepsilon) = \left\{ \exists u \in [0,1] : \begin{array}{ll} |Y(t) {+}t| \leq \varepsilon  &\forall t \in [0,  u],\\
|Y(t)-(1 {-}t)| \leq \varepsilon &  \forall t \in [u+ \varepsilon, 1] { \cap [0,1]}
\end{array} \right\}.$$ Applying the Vervaat transformation to $ X^{ \mathrm{br}, ( \alpha)}$, we deduce from \cref{cv:Xeps} that for every $ \varepsilon>0$ we have   \begin{eqnarray} \label{eq:Jbra}  \mathbb{P} \left(\mathsf{J}( X^{ \mathrm{br}, (\alpha)}, \varepsilon)\right) & \xrightarrow[ \alpha \downarrow1]{} & 1.  \end{eqnarray}

We then rely on the following result: \begin{lemma}  \label{lem:dur}
For every $ \alpha \in (1,2)$, let $(B^{(\alpha)}_{t}; 0 \leq t \leq 1)$ be a cˆdlˆg process with $0=B_{0}^{(\alpha)} = B_{1}^{(\alpha)}$ and such that the following two conditions hold:
\begin{enumerate}
\item[(i)] For every $ \varepsilon>0$, we have $\mathbb{P}\big(  \mathsf{J}( B^{(\alpha)}, \varepsilon)\big) \to 1$ as $ \alpha \downarrow 1$;
\item[(ii)] For every $\alpha \in (1,2)$ and every $n \geq 1$, the increments $$\left\{ \big(B^{(\alpha)}_{t+ i/n}-B^{(\alpha)}_{i/n}\big)_{0 \leq t \leq 1/n} : {0\leq i\leq n-1}\right\}$$ are exchangeable.
\end{enumerate}
Then  \begin{eqnarray} \label{eq:goallem}B^{(\alpha)} \quad \xrightarrow[\alpha\downarrow 1]{(d)} \quad \left(\mathbbm {1}_{  \{U \leq t\}}-t; \, 0 \leq t \leq 1 \right),   \end{eqnarray}
where the convergence holds in distribution for the Skorokhod topology on $ \D([0,1], \R)$ and where $U$ is an {independent}  uniform variable over $[0,1]$.
\end{lemma}

If we assume for the moment this lemma, the proof of \cref{prop:cvsautexc} is completed as follows. The LŽvy bridges $  X^{ \mathrm{br},(\alpha)}$ satisfy the assumptions of Lemma \ref {lem:dur}. Indeed, (i) is satisfies thanks to \eqref{eq:Jbra} and (ii) follows from absolute continuity. Lemma \ref {lem:dur} entails that $X^{ \mathrm{br},(\alpha)} \to \left(\mathbbm {1}_{  \{U \leq t\}}-t; \, 0 \leq t \leq 1 \right)$ the convergence holds in distribution for the Skorokhod topology as $ \alpha \downarrow 1$. It then suffices to apply the Vervaat transform to the latter convergence to get the desired result. \endproof

It remains to establish Lemma \ref {lem:dur}.

\proof[Proof of Lemma \ref{lem:dur}]  Fix $\alpha \in (1,2)$ and $n \geq 1$. We introduce the following shuffling operation on $B^{(\alpha)}$: cut the bridge $B^{(\alpha)}$ into $n$ pieces between times $[i/n, (i+1)/n]$ for $0 \leq i \leq  n-1$. Then ``shuffle'' these $n$ pieces uniformly at random, meaning that these $n$ pieces are concatenated after changing their order by using  an independent uniform permutation of $\{1, 2, \ldots, n\}$. Denote by $ \widetilde{B}^{(\alpha),n}$ the  process obtained in this way. Assumption (ii) garantees that $ \widetilde{B}^{(\alpha),n}$ has the same distribution as $ B^{(\alpha)}$. In particular, for every $ \epsilon >0$, $ \mathbb{P}( \mathsf{J}(\widetilde{B}^{(\alpha), n},  \varepsilon)) \to 1$ as $\alpha \downarrow 1$, uniformly in $n$.

\textsc{First step: at most one large jump.} We first show that for every $ \delta >0$, the probability that there are two jumps in $B^{(\alpha)}$ larger than $\delta$ tends to $0$ as $ \alpha \downarrow 1$. To this end, argue by contradiction and assume that there exists $ \eta >0$ such that  along a subsequence $\alpha_{k} \downarrow 1$  with probability at least $\eta$ the bridge $B^{(\alpha_{k})}$ has two jump times $T_{1}^{({k})} \ne T_{2}^{({k})}$ at which $ \Delta ^{(\alpha_{k})}$ is greater than $\delta$. Now, choose $n_{k} \to \infty$ so that 
$$  \Pr{  \left | T_{1}^{({k})} - T_{2}^{({k})} \right | > 1/n_{k}}  \xrightarrow[k\to\infty]{} 1.$$
But, conditionally on the event $\{  | T_{1}^{({k})} - T_{2}^{({k})}  | > 1/n_{k}\}$ , with probability tending to one as $k \rightarrow \infty$, these two jumps will fall in different time intervals of the form $[i/n_{k}, (i+1)/n_{k}]$ in the shuffled process $B^{(\alpha_{k}),{n_k}}$.  Hence, we deduce that with probability asymptotically larger than  $\eta/100$ (this value is not optimal), there exist two jump times $\widetilde{T}_{1}^{(k)}$ and  $\widetilde{T}_{2}^{(k)}$ of $ \widetilde{B}^{(\alpha_{k}),n_{k}}$ such that 
$$ |\widetilde{T}_{1}^{(k)} - \widetilde{T}_{2}^{(k)}| \geq  \frac{1}{3} \qquad \mbox{ and } \qquad \Delta \widetilde{B}^{(\alpha_{k}), n_{k}}_{\widetilde{T}_{1}^{(k)}} \geq \delta, \quad \Delta \widetilde{B}^{(\alpha_{k}), n_{k}}_{\widetilde{T}_{2}^{(k)}} \geq \delta.$$
If one chooses $ \varepsilon \in (0, \delta \wedge 1/4)$, this contradicts the fact that $ \mathbb{P}( \mathsf{J}(\widetilde{B}^{(\alpha_{k}), n_{k}},  \varepsilon)) \to 1$ as $k \to \infty$.

\textsc{Second step: one jump of size roughly $1$.} We only sketch the argument and leave the details to the reader. Denote by $T_{\alpha}$ the time when
$B^{(\alpha)}$ achieves its largest jump. Let $ \alpha_{k}$ be a sequence such that $\alpha_{k} \downarrow 1$ as $k \rightarrow \infty$. Let $0 \leq I_{k} \leq n_{k}-1$ be the integer such that $T_{\alpha_{k}} \in [I_{k}/n_{k}, (I_{k}+1)/n_{k}]$, and set $$\delta_{k} \quad := \quad B^{(\alpha_{k})}_{(I_{k}+1)/n_{k}}- B^{(\alpha_{k})}_{I_{k}/n_{k}}.$$  Then let $n_{k} \to \infty$ be a sequence of integers such that the following three converges hold in probability as $k \rightarrow \infty$:
\begin{enumerate}
\item[(i)] $ \displaystyle |\delta_{k}-\Delta B^{(\alpha_{k})}_{T_{\alpha_{k}}}|  \quad\mathop{\longrightarrow}^{(\P)}_{k \rightarrow \infty} \quad 0;$
\item[(ii)] $ \displaystyle\frac{\delta_{k}^2 \vee 1}{n_{k}} \quad\mathop{\longrightarrow}^{(\P)}_{k \rightarrow \infty} \quad 0;$
\item[(iii)] $ \displaystyle\sup_{ i \ne I_{k}}\sup_{0 \leq t \leq 1/n_{k}}
  \left|B^{(\alpha_{k})}_{t+i/n_{k}} - B^{(\alpha_{k})}_{i/n_{k}}\right|   \quad\mathop{\longrightarrow}^{(\P)}_{k \rightarrow \infty} \quad  0$.
  
Indeed, this is possible since, by the first step, we know that all the jumps of $B^{(\alpha_{k})}$, its largest jump excluded, converge in probability to $0$ as $k \rightarrow \infty$.
\end{enumerate}

Denote by $ \widehat{B}^{(\alpha_{k}),n_k}$ the function on $[0,1]$ obtained by doing a random shuffle of $B^{(\alpha_k)}$ of length $1/n_{k}$ after discarding the time interval  that contains $T_{\alpha_{k}}$, and then scaling time by a factor $n_{k}/(n_{k}-1)$ so that $ \widehat{B}^{(\alpha_k),n_k}$ is defined on $[0,1]$. The proof is completed if we manage to check that $ \widehat{B}^{(\alpha_{k}),n_{k}}$ converges in probability towards the function $ t \mapsto -t$ and $ \delta_{k} \to 1$ in probability.

To do so, let us introduce the empirical variance of the small increments 
$${\Sigma}_{k} := \sum_{0 \leq i \ne I_{k} \leq n_{k-1}}\left|B^{(\alpha_{k})}_{(i+1)/n_{k}} - B^{(\alpha_{k})}_{i/n_{k}}\right|^2.$$
We shall first establish that   $ {\Sigma}_{k} \to 0$ in probability as $k \rightarrow \infty$. To this end, suppose by contradiction that $\Sigma_{k}$ does not converge to $0$ in probability as $k \rightarrow \infty$. Then, up to extraction, there exists a fixed $c>0$ such that $ \mathbb{P}( \Sigma_{k} \geq c) \geq c$ for every $k$ large enough. Then consider  the family of $n_{k}-1$ increments 
$$ \left\{X_{i,k} \right\}_{0 \leq i \ne I_{k} \leq n_{k}-1} := \left\{ B^{(\alpha_{k})}_{(i+1)/n_{k}}-B^{(\alpha_{k})}_{i/n_{k}} + \frac{\delta_{k}}{n_{k}-1} \right\}_{0 \leq i \ne I_{k} \leq n_{k}-1}.$$
Observe that we have 
\begin{equation}
\label{eq:cv3} \sum_{0 \leq i \ne I_{k} \leq n_{k}-1} X_{i,k} = 0, \qquad   \sup_{i \ne I_{k}} \left |X_{i,k} \right|   \quad\mathop{\longrightarrow}^{(\P)}_{k \rightarrow \infty} \quad 0, \qquad \sum_{0 \leq i \ne I_{k} \leq n_{k}-1} X_{i,k}^2 - \Sigma_{k}  \quad\mathop{\longrightarrow}^{(\P)}_{k \rightarrow \infty} \quad 0.
\end{equation}
For the second and third convergences, we use (ii) and (iii). 

Then let $ \pi :  \{ 1,2, \ldots, n_{k}-1\} \rightarrow \{ 0,1, \ldots, n_{k}-1\} \backslash \{ I_{k}\}$ be a uniform bijection, and define the random continuous function $ \overline{B}^{(\alpha_k),n_k}$ on $[0,1]$ by linearly interpolating between the points of coordinates $ (0,0)$, $ \left( \frac{i}{n_{k}-1}, X_{ \pi(i),k}\right)$ for $1 \leq i \leq n_{k}-1$. 
From \cite[Theorem 24.2]{Bil68} and \eqref{eq:cv3}, it follows that, on the event $ \{ \Sigma_{k} \geq c\}$, the random function 
$$  \left(  \frac{\overline{B}^{(\alpha_{k}),n_{k}}_{t}}{ \sqrt { \Sigma_{k}}}; 0 \leq t \leq 1 \right)     $$
 converges in distribution towards a standard Brownian bridge of variance $1$. By (iii), the previous distributional convergence also holds when $ \overline{B}^{(\alpha_{k}),n_{k}}$ is replaced by $( \widehat{B}^{(\alpha_{k}), n_{k}}_{t} + \delta_{k} t )_{0 \leq t \leq 1}$. A moment's though shows then the condition $ \mathbb{P}( \mathsf{J}(\tilde{B}^{(\alpha_{k}),n_{k}}, \varepsilon)) \to 1$ cannot be satisfied and hence that  $ {\Sigma}_{k} \to 0$ in probability.

  Then the proofs of \cite[Theorems 24.1 and 24.2]{Bil68} give that the random function $$ \left(\overline{B}^{(\alpha_{k}),n_{k}}_{t} \ ; \ 0 \leq t \leq 1 \right) $$ converges in probability towards the constant function equal to $0$ on $[0,1]$, denoted by $ \mathbf {0}$. As before, using (iii), we deduce that $( \widehat{B}^{(\alpha_{k}), n_{k}}_{t} + \delta_{k} t )_{0 \leq t \leq 1}$ in turn converges to $ \mathbf {0}$  in probability. Using the fact that $ \mathsf{J}( \tilde{B}^{(\alpha_{k}),n_{k}}, \varepsilon) \rightarrow 1$ as $k \rightarrow \infty$, we get that  
  $\delta_{k} \to 1$ in probability. Using (i), this implies that $\Delta B_{T_{\alpha_{k}}}^{(\alpha_{k})} \to 1$ in probability. It follows that $ \widehat{B}^{(\alpha_{k}), n_{k}}$ indeed converges to $t \mapsto -t$ in probability. The details are left to the reader. \endproof

\subsubsection{Others lemmas}

Denote by $ \oD(Y)$ the size of the largest jump of a cˆdlˆg function $Y$. This quantity is of interest since by construction the length of the longest cycle in the stable looptree $ \La$ is equal to $\oD( X^{ \mathrm{exc}, (\alpha)})$.

\begin {proposition} \label{prop:longueur}We have: $$\Es{\oD( X^{ \mathrm{exc}, (\alpha)})}= \Gamma \left(1- \frac{1}{ \alpha} \right) \beta ,$$
where $ \beta>0$ is the unique solution to the equation
$$ \sum_{n=0}^ \infty \frac{(-1)^n  \beta^n}{(n- \alpha) n!}=0.$$
\end{proposition}

Setting $f( \beta)=\sum_{n=0}^ \infty {(-1)^n  \beta^n}/({(n- \alpha) n!})$, note that existence and uniqueness of this solution follow for instance from the fact that $f$ is continuous, increasing, $f(0{+})<0$ and $f(1)>0$.

\proof Recall the scaling properties of the It™ measure  $\underline{n}^{(\alpha)}$ from \cref{sec:absolutecontinuity}. 
Our main ingredient is a result of Bertoin \cite[Corollary 2]{Ber11}, which identifies the distribution of the maximal jump ${ \Delta^*}$ under the excursion measure $ \underline{n}^{(\alpha)}$
$$ \underline{n}^{(\alpha)}(  { \Delta^*}>x) = \beta/x, \qquad x>0.$$
Then to calculate $\Es{ \oD( X^{ \mathrm{exc}, (\alpha)})}$ it suffices to write 
$$ \begin{array}{rcl}
 \beta \Gamma(1-1/\alpha)
 &=& \displaystyle \underline{n}^{(\alpha)} ( \oD>1)  \Gamma(1-1/\alpha) \\
 &=& \displaystyle\int_0 ^ \infty \underline{n}^{(\alpha)}_ {(a)}(\oD>1) \frac{da}{\alpha a^{1/\alpha+1} }   \qquad \textrm {by property $(iii)$ in \cref{sec:excmin}}\\
&=&  \displaystyle\int_0 ^ \infty \underline{n}_ {(1)}^{(\alpha)} \left( \oD > \frac{1}{a^  {1/ \alpha}} \right)  \frac{da}{\alpha
a^{1/\alpha+1}}   \qquad    \textrm {by property $(ii)$ in \cref{sec:excmin}}\\
&=& \displaystyle\int_0 ^ \infty \underline{n}_ {(1)}^{(\alpha)} \left( \oD > u \right)  \frac{u^ {1+ \alpha}}{\alpha }  \cdot \frac{ \alpha du}{u ^ {1+ \alpha}} \qquad \textrm {by change of variables} \\
&=&   \displaystyle\int_0 ^ \infty  \underline{n}^{(\alpha)}_ {(1)} \left( \oD > u \right) du = \Es { \oD( X^{ \mathrm{exc}, (\alpha)})}  \qquad   \textrm {by definition}. \qedhere
\end{array} $$

Note that $\Es{\oD( X^{ \mathrm{exc}, (\alpha)})}$ converges towards $1$ as $ \alpha \downarrow 1$ and towards $0$ as $ \alpha \uparrow 2$. This is consistent with \cref{prop:cve,prop:cvsautexc}.

\begin{remark} \label{rem:jump1} Janson \cite[Formula (19.97)]{Jan12} gives the cumulative distribution function of $\oD( X^{ \mathrm{exc}, (\alpha)})$:
\begin{eqnarray*}
&& \Pr{\oD( X^{ \mathrm{exc}, (\alpha)}) \leq u} \\
&&   =\frac{| \Gamma(-1/ \alpha)|}{2\pi}  \int_{ - \infty}^{ + \infty}
\exp \left( \frac{1}{ \Gamma(- \alpha)}  \left( 
\int_0^u  x^{-\alpha-1} \left( e^{i t x}-1- i t x \right) \,\mathrm{d} x 
-\frac{u^{-\alpha}}{\alpha}-i t\frac{u^{1-\alpha}}{\alpha-1} \right) \right)\,\mathrm{d} t
\end{eqnarray*}
where $u \geq 0$. However, it seems difficult to calculate  $\Es{\oD( X^{ \mathrm{exc}, (\alpha)})}$ using this formula. Note also that if one manages to use this explicit expression to prove that $\oD( X^{ \mathrm{exc}, (\alpha)}) \to 1$ in probability as $ \alpha \downarrow 1$, this would simplify the proof of \cref{thm:1and2} $(i)$.
\end{remark}

\begin {lemma} \label{lem:technique}Let $p^{( \alpha)}_t$ be the density of the law of $X_t^{( \alpha)}$. There exist a constant $C>0$, which does depend on $ \alpha$, such that:
$$ \forall \alpha \in (3/2,2], \quad \forall x \in \R, \qquad p^{( \alpha)}_{1}(x) \leq C.$$
\end {lemma} 

\begin{proof}The characteristic function $ \phi^{( \alpha)}$ of $X_1^{( \alpha)}$ is given by \cite[Theorem C.3.]{Zol86}$$ \phi^{( \alpha)}(t)= \Es { \exp \left(i t X_1^{( \alpha)} \right)}= \exp \left( -  \left| \cos \left( \frac{ \pi \alpha}{2} \right) \right| t^{ \alpha} - i \sin \left( \frac{ \pi \alpha}{2} \right) t^{ \alpha}\right), \qquad t \geq 0.$$
For $ x \in \R$, by the inversion formula $p^{( \alpha)}_{1}(x)= ( 2  \pi)^{-1} \int_{ \R} e^{-itx} \phi^{( \alpha)}(t) dt$, we get
$$ \left|p^{( \alpha)}_{1}(x) \right| \leq  \frac{1}{2 \pi} \int_{ \R} |\phi^{(\alpha)}(t)|dt= \frac{ \Gamma( 1/ \alpha)}{ \pi \alpha \left| \cos \left( \frac{ \pi \alpha}{2} \right) \right|^{1/ \alpha}}.$$
The conclusion immediately follows.
\end{proof}
 
\subsection{Limiting cases $\alpha \downarrow 1$ and  $\alpha \uparrow 2$}
\label{sec:limitcases}
In this section, we keep the notation  $X^{(\alpha)},X^{ \mathrm{br},(\alpha)},X^{\exc,(\alpha)}$  for respectively  the $ \alpha$-stable spectrally positive process, its bridge and its normalized excursion.

We prove \cref{thm:1and2} concerning the limiting behavior of $ \La$ as $  \alpha \downarrow 1$ and $ \alpha \uparrow 2$.  Since $ \La$ is coded by $X^{ \mathrm{exc}, (\alpha)}$, it should not be surprising that these results are consequences of \cref{prop:cve,prop:cvsautexc} which describe the limiting behavior of $X^{ \mathrm{exc}, (\alpha)}$ as $  \alpha \downarrow1$ and $ \alpha \uparrow 2$. We will see this is indeed the case when $ \alpha \rightarrow1$, but that some care is needed when $ \alpha \rightarrow 2$ because of the presence of an additional factor $ \frac{1}{2}$.

Before proving \cref{thm:1and2} we briefly recall the definition of the Gromov--Hausdorff topology. We refer to \cite{BBI01} for additional details.

\paragraph{The Gromov--Hausdorff topology.} If $(E,d)$ and $(E',d')$ are two  compact  metric spaces, the Gromov--Hausdorff distance between
${E}$ and ${E'}$  is 
 \begin{eqnarray*} \op{d_{GH}}({E},{E'}) &=& \inf\big\{\op{d}_{\op{H}}^F(\phi(E),\phi'(E'))\big\}, \end{eqnarray*} where the infimum is taken over all choices of the metric space $(F,\delta)$ and of the  isometric embeddings $\phi : E \to F$ and $\phi'  : E' \to F$ of $E$ and $E'$ into $F$ and $ \mathrm{d}_{ \mathrm{H}}^F$ is the Hausdorff distance between compacts sets in $F$. An alternative definition of this distance uses \textit{correspondences.} A correspondence between two  metric spaces $(E,d)$ and $(E',d')$ is a subset $\mathcal{R}$ of $E\times E'$ such that, for every $x_{1} \in E$, there exists at least one point $x_{2}\in E'$ such that $(x_{1},x_{2}) \in \mathcal{R}$ and conversely, for every $y_{2}\in E'$, there exists at least one point $y_{1}\in E$ such that $(y_{1},y_{2}) \in \mathcal{R}$. The distortion of the correspondence $\mathcal{R}$ is defined by 
 $$ \op{dis}(\mathcal{R}) = \sup_{}\big\{|d(x_{1},y_{1})-d'(x_{2},y_{2})| : (x_{1},x_{2}),(y_{1},y_{2}) \in \mathcal{R} \big\}.$$ The Gromov--Hausdorff distance can be expressed in terms of correspondences by the formula
 \begin{equation}
 \label{GHcorres}
 \op{d_{GH}}({E},{E'})= \frac{1}{2} \inf \big\{\hspace{-0.5mm}\op{dis}(\mathcal{R})\big\},
 \end{equation} where the infimum is over all correspondences $\mathcal{R}$ between ${E}$ and ${E'}$. 
The  Gromov--Hausdorff  distance is indeed a metric on the space  of all isometry classes of compact metric spaces, making it separable and complete. 

\medskip

\proof[Proof of \cref{thm:1and2}] {Recall the notation of \cref{sec:definition}.} Assertion $(i)$ in an immediate consequence of  \cref{prop:cvsautexc}. {Indeed, \cref{prop:cvsautexc} implies that as $ \alpha \downarrow 1$, the sequence of functions $ (s,t) \mapsto \delta_{{s \wedge t}}\big( {x_{s\wedge t }^ {s}}, {x_{s\wedge t }^ {t}}\big)$ converges  in probability towards the function $(s,t) \mapsto |s-t|$, uniformly on $[0,1]^2$, while the  sequences of functions $ (s,t) \mapsto {d}_{0}( s \wedge t,s)$ and ${d}_{0}( s \wedge t,t)$ converge in probability towards the constant function equal to $0$, uniformly on $[0,1]^2$. By \eqref{eq:def2}, this implies that $ (s,t) \mapsto \mathrm{d}^{(\alpha)}(s,t)$ converges  in probability towards $(s,t) \mapsto |s-t|$, uniformly on $[0,1]^2$, implying (i). We leave details to the reader.}

We now establish $(ii)$. Recall from \eqref{def:dh} the definition of the pseudo-distance $ \mathrm{d}_{h}$ for a function $h : [0,1] \to \mathbb{R}_{+}$. We will prove that we have the following convergence in distribution 
 \begin{eqnarray*}  \mathrm{d}^{(\alpha)}( \cdot, \cdot) &\xrightarrow[\alpha\to2]{(d)}& \frac{\sqrt{2}}{2} \cdot  \mathrm{d}_{ \mathbf{e}}(\cdot, \cdot)  \end{eqnarray*} for the uniform norm over $[0,1]^2$, which in turn will imply $(ii)$. We first check that the sequence of random pseudo-distances $ (\textrm{d}^{ (\alpha)})$ is tight as $\alpha \to 2$ for the uniform topology on $[0,1]^2$. Fix $ \epsilon >0$. By \cite[Theorem 7.3]{Bil99} (this reference covers the case of $[0,1]$ but the extension to $[0,1]^{2}$ is straightforward), it is sufficient to check that there exists $ \eta>0$ such that for $\alpha$ sufficiently close to $2$ we have
\begin{equation}
\label{eq:unif} \Pr { \sup_{|x-y| < \eta} \mathrm{d}^{ (\alpha)}(x,y)> \epsilon }< \epsilon.
\end{equation}

Note that by \cref{prop:cve}, the pseudo-distance
$\mathrm{d}_{ X^{\exc,(\alpha)}}( \cdot, \cdot)$ converges in distribution for the uniform norm on $[0,1]^2$ towards $\sqrt{2} \cdot \mathrm{d}_{ \mathbf{e}}( \cdot, \cdot)$ as $\alpha \uparrow 2$. It follows that there exists $ \eta>0$ such that for $ \alpha$ sufficiently close to $2$
\begin{equation}
\label{eq:unif2} \Pr { \sup_{|x-y| < \eta}  \mathrm{d}_{ X^{\exc,(\alpha)}}(x,y)> \epsilon }< \epsilon.
\end{equation}
But, by \cref{lem:controls} $(ii)$ for every $x,y\in [0,1]$ we have
$\mathrm{d}^{ (\alpha)}(x,y) \leq \mathrm{d}_{ X^{\exc,(\alpha)}}(x,y).$ Our claim \eqref{eq:unif} then follows from \eqref{eq:unif2}.

Since  $ (\textrm{d}^{ (\alpha)})$ is tight as $\alpha \to 2$ and since $\mathrm{d}_{ X^{\exc,(\alpha)}}( \cdot, \cdot)$ converges in distribution towards $\sqrt{2} \cdot \mathrm{d}_{ \mathbf{e}}( \cdot,\cdot)$, a density and continuity argument shows that in order to identify  the limit of any convergent subsequence of $ (\textrm{d}^{(\alpha)})$, by \cite[Theorem 7.3]{Bil99} (this reference covers the case of $[0,1]$ but the extension to $[0,1]^{2}$ is straightforward), it is sufficient to check that 
\begin{equation}
\label{eq:reroot}\frac{d^{(\alpha)}(U,V)}{  \mathrm{d}_{ X^{\exc,(\alpha)}}(U,V)} \quad \xrightarrow[\alpha \uparrow 2]{( \P)} \quad \frac{1}{2}.
\end{equation}
where $U,V$ are independent random uniform variables on $[0,1]$.   We claim that it suffices to prove {that}
 \begin{eqnarray} \frac{d^{(\alpha)}(0,U)}{X^{\exc,(\alpha)}(U)} & \xrightarrow[\alpha\uparrow2]{( \P)}& \frac{1}{2}.  \label{eq:ingredient}\end{eqnarray}
Indeed, the reader may either strengthen the following proof by splitting at the most common ancestor $U \wedge V$, or invoke a re-rooting property of $X^{ \exc, (\alpha)}$ at a uniform location which {gives} $$\Big (\mathrm{d}^{(\alpha)}(U,V), \mathrm{d}_{X^{\exc,(\alpha)}}(U,V)\Big) \quad  \overset{(d)}{=}\quad \Big( \mathrm{d}^{(\alpha)}(0,U), \mathrm{d}_{X^{\exc,(\alpha)}}(0,U)\Big),$$ see \cref{rem:reroot}. 
We now establish \eqref{eq:ingredient}. For a cˆdlˆg function  $Y \in \mathbb {D}([0,1]), \R)$ recall the notation $x_{s}^t(Y) ,u_{s}^t(Y)$ from \cref{sec:descents} and for $0 \leq  \eta \leq t \leq 1$ set $$ \textsf{Q}_{ \eta}^{t}(Y)= \frac{  \displaystyle \sum_{  \eta \leq s, \  s \preccurlyeq t} \Delta Y_s\min\big( u_{s}^t( Y),1-u_{s}^t(Y)\big)}{ \displaystyle \sum_{ \eta \leq s, \  s \preccurlyeq t} \Delta Y_s u_{s}^t( Y)}.$$
By \eqref{eq:u} and \eqref{eq:u2}, we have:
$$ \frac{d^{(\alpha)}(0,U) }{X^{\exc,(\alpha)}(U)} \quad=	 \quad{\textsf{Q}_{0}^{U}(X^{\exc,(\alpha)})}.$$
By using the Vervaat transformation (recall \cref{sec:absolutecontinuity}), we get that \begin{equation}
\frac{d^{(\alpha)}(0,U) }{X^{\exc,(\alpha)}(U)} \quad
\displaystyle \mathop{=}^{(d)} \quad{\textsf{Q}_{0}^{1}(X^{ \br,(\alpha)})} \label{eq:quotientdegueu}.
\end{equation}
 It is thus sufficient to show that the last quantity   converges in probability to $ {1}/{2}$ as $ \alpha \uparrow 2$. As usual, we  replace  the bridge $ X^{\br,(\alpha)}$ by the  $\alpha$-stable process $ X^{(\alpha)}$ and first prove that
 \begin{eqnarray} \label{eq:pourprocessus}{\textsf{Q}_{0}^{1}(X^{ (\alpha)})}  & \xrightarrow[\alpha \uparrow 2]{( \P)} & \displaystyle \frac{1}{2}.  \end{eqnarray}
To this end, note that by \cref{prop:itomeasure}, the collection $\{u_{s}^1( X^{(\alpha)}) : s \in [0,1], s \preccurlyeq 1\}$  is an i.i.d.\,collection of uniform variables also independent of $ \{\Delta X^{(\alpha)}_{s}, :  0 \leq s,  s \preccurlyeq 1\}$.  By \cref{lem:jumps} we have $$ \sum_{0 \leq s, \  s \preccurlyeq 1} \Delta X^{(\alpha)} _{s} \geq \sum_{  0 \leq s,  s \preccurlyeq 1} x_s^1(X^{ (\alpha)})= X_{1}^{( \alpha)}- \inf_{[0,1]} X^{( \alpha)} \quad \xrightarrow[\alpha \uparrow2]{(d)} \quad \sqrt{2} \cdot ( B_{1} - \inf_{[0,1]} B).$$
On the other hand, we have  for $  \varepsilon >0$
 \begin{eqnarray*} 
 \Pr{\sup_{s \in [0,1]} \Delta X^{(\alpha)}_{s} \geq   \varepsilon} \quad  = \quad  1-\exp \left( - \Pi^{(\alpha)}([ \varepsilon, \infty)) \right)  \end{eqnarray*}
 which converges to $0$ as $ \alpha \uparrow 2$ {by}  \eqref{eq:levymeasure}. {Setting $ \mathcal{S}=   \{\Delta_{s}(X^{(\alpha)}) ; \,  0 \leq s, s \preccurlyeq 1\}$, it follows that $ \sup \mathcal{S}$  converges in probability towards $0$ as  $\alpha \uparrow2$, and the sum of all the elements of $ \mathcal{S}$ converges in probability towards a positive random variable as $\alpha \uparrow2$.} 
 We are thus in position to apply a classic weak law of large numbers (for example by using an $L^2$ estimate) and get {the following two convergences}:
 \begin{eqnarray*} \frac{ \displaystyle\sum_{ 0 \leq s, \ s \preccurlyeq 1} \Delta  X^{(\alpha)} _{s}  \cdot u_{s}^1( X^{(\alpha)})}{\displaystyle\sum_{ 0 \leq s, \ s \preccurlyeq 1} \Delta  X^{(\alpha)} _{s}}   & \xrightarrow[\alpha\uparrow2]{( \mathbb{P})}&  \Es{U}=1/2, \\ \frac{\displaystyle\sum_{ 0 \leq s, \ s \preccurlyeq 1} \Delta  X^{(\alpha)} _{s} \min\big( u_{s}^t( X^{(\alpha)}),1-u_{s}^t( X^{(\alpha)})\big) }{\displaystyle\sum_{ 0 \leq s, \ s \preccurlyeq 1} \Delta  X^{(\alpha)} _{s}}  & \xrightarrow[\alpha\uparrow2]{( \mathbb{P})}&\Es { \min(U,1-U)}=1/4.  \end{eqnarray*}
This proves \eqref{eq:pourprocessus}.

We now complete the proof of \eqref{eq:ingredient} by 
showing that
\begin{eqnarray} \label{eq:pourbridge}{\textsf{Q}_{0}^{1}(X^{ \br, (\alpha)})}  & \xrightarrow[\alpha \uparrow 2]{( \P)} & \displaystyle \frac{1}{2}  \end{eqnarray}
by using an absolute continuity argument. For a cˆdlˆg function  $Y \in \mathbb {D}([0,1], \R)$,  set $u_{ \star}(Y)=\inf\{ t \in [0,1]; \min(Y(t-), Y(t))= \inf_{[0,1]} Y\}$. 
 Fix $ \epsilon >0$. We claim that there exists $ \eta \in (0,1)$ such that for every $\alpha \in (1,2)$ sufficiently close to $2$ we have
$$\Pr {\textsf{Q}^1_{ 0}(X^{\br,(\alpha)}) \neq \textsf{Q}^1_{ \eta}(X^{\br,(\alpha)})} \leq  \epsilon $$Indeed, notice first that $ \textsf{Q}^1_{ 0}(Y)=\textsf{Q}^1_{ u_{ \star}(Y)}(Y)$ and second that $u_{ \star}(X^{\br,(\alpha)})$ is uniformly distributed on $[0,1]$ (
see \cite[VIII, Exercise 6]{Ber96}).
Next, by absolute continuity (see \cref{sec:normexc}) applied to the dual process $t \mapsto X_{1}-X_{(1-t)-}$,
$$
 \Pr { \left|\textsf{Q}^1_{ \eta}(X^{\br,(\alpha)})-1/2 \right| > \delta} = \Es { \mathbbm {1}_{ \left|\textsf{Q}^1_{ \eta}({X}^{(\alpha)})-1/2 \right| > \delta} \frac{p^{( \alpha_n)}_{\eta}(X^{(\alpha)}_{ \eta} )}{ p^{( \alpha)}_{1}(0)}}.$$
Since the densities $p^{( \alpha)}_t$ enjoy the scaling relation $p^{( \alpha)}_t(x)=t^{-1/ \alpha} p^{( \alpha)}_1(x t^{-1/ \alpha})$ by \cref{lem:technique}, it follows that there exists a constant  $C>0$ (depending on $\eta$) such that, for every $\alpha \in (\frac{3}{2},2)$,
$$ \Pr { \left|\textsf{Q}^1_{ \eta}(X^{\br,(\alpha)})-1/2 \right| > \delta}  \leq  C \Pr { \left|\textsf{Q}^1_{ \eta}({X}^{(\alpha)})-1/2 \right| > \delta}.$$ Thus, putting the pieces together, for every $\alpha$ sufficiently close to $2$ we have
$$\Pr { \left|\textsf{Q}^1_{ 0}(X^{\br,(\alpha)})-1/2 \right| > \delta} \leq   \Pr {\textsf{Q}^1_{ 0}(X^{\br,(\alpha)}) \neq \textsf{Q}^1_{ \eta}(X^{\br,(\alpha)})} +C \Pr { \left|\textsf{Q}^1_{ \eta}({X}^{(\alpha)})-1/2 \right| > \delta}.$$
A minor adaptation of \eqref{eq:pourprocessus} shows that
$\textsf{Q}^1_{ \eta}(X^{\br,(\alpha)})$ converges in probability to $\frac{1}{2}$ as $\alpha \uparrow 2$. This completes the proof of \cref{thm:1and2} $(ii)$.\endproof

\subsection{Hausdorff dimension of looptrees}

In this section, we study fractal properties of looptrees, and prove in particular \cref{thm:dimension} which identifies the Hausdorff dimension of $ \La$ (see  \cite[Sec.~4]{Mat95} for the definition and background on Hausdorff dimension). {Recall the definition of $ \mathscr{L}_{\alpha}$ using $\X$ in \cref{sec:definition}. In this section, the dependence of $ \X$ in $\alpha$ is implicit.}

\subsubsection{Upper bound}
\proof We construct a covering of $ \mathscr{L}_{\alpha}$ as follows. Fix $ \varepsilon>0$ and let $(t_{i}^{( \varepsilon)})_{1 \leq i \leq N _{ \varepsilon}}$ be an increasing enumeration of the elements of the finite set  $\{t \in [0,1]; \, \Delta_{t}> \varepsilon^{ 1/\alpha} \}$ and set $ t_{0}^{( \varepsilon)}=0$ and $ t_{N_{ \varepsilon}+1}^{( \varepsilon)}=1$. Recall that $\textbf{p}: [0,1] \rightarrow \La$ is the canonical projection. It is clear that
$$ \bigcup_{i=0}^{N_{ \varepsilon}} \mathbf{p}([t_{i}^{( \varepsilon)},t_{i+1}^{( \varepsilon)}))$$
is a covering of $ { \La}$. By \cref{lem:controls} $(ii)$, we have   \begin{eqnarray} \mathsf {Diam}\Big(\mathbf{p}\big([t_{i}^{( \varepsilon)},t_{i+1}^{( \varepsilon)})\big)\Big) &\leq&  {2 \cdot } \mathsf{Amp}_{[t_{i}^{( \varepsilon)},t_{i+1}^{( \varepsilon)})} X^{ \mathrm{exc}},   \label{eq:diamdiam} \end{eqnarray} where by definition
$$ \mathsf{Diam}(A) := \sup\{ d(u,v) : u,v \in A\} \qquad \textrm{and} \qquad \mathsf{Amp}_{[s,t]}f := \sup\{ |f(x)-f(y)| : x,y \in [s,t]\}.$$
  We shall now prove that, for every $\eta \in (0,1/ \alpha)$,
\begin{equation}
\label{eq:vers1} \lim_{\varepsilon \to 0 } \Pr{N_{ \varepsilon} \leq  \varepsilon^{-1-\eta} \quad  \mbox{ and } \quad \mathsf{Amp}_{[t_{i}^{( \varepsilon)},t_{i+1}^{( \varepsilon)})} X^{ \mathrm{exc}} \leq \varepsilon^{1 / \alpha - \eta}, \quad \forall i \leq N_{ \varepsilon}} \quad = \quad 1.
\end{equation} This will entail that a.s.\,$ \dim_{H}( \La) <  \alpha(1+ \eta)/(1- \eta \alpha)$, implying the a.s.\,upper bound  $ \dim_{H}\left(\mathscr{L}_\alpha\right) \leq  \alpha$  since $\eta \in (0, 1/ \alpha)$ was arbitrary. 

Instead of proving \eqref{eq:vers1} directly, we will {first} prove a similar statement involving the unconditioned process $X$. 
 Let $(t_{i}^{(\varepsilon),*})_{i \geq 1}$ be an increasing enumeration of the times where $X$ makes a jump  larger than $ \varepsilon^{1/\alpha}$ (with the convention $t_{0}^{( \varepsilon),*}=0$), and set $N_{\varepsilon}^* = \#\{ i \geq 1 : t_{i}^{( \varepsilon),*} \leq 1 \}$.  By standard arguments involving continuity relations between $X$  and the LŽvy bridge $ X^ \br$ as well as the Vervaat transformation between $ X^ \br$ and $X^{ \mathrm{exc}}$ (see \cref{sec:absolutecontinuity}),  \eqref{eq:vers1} holds if we manage to prove that
\begin{equation}
\label{eq:vers1bis} \lim_{\varepsilon \to 0 } \Pr{N^*_{ \varepsilon} \leq  \varepsilon^{-1-\eta} \quad  \mbox{ and } \quad \mathsf{Amp}_{[t_{i}^{( \varepsilon),*},t_{i+1}^{( \varepsilon),*})} X \leq \varepsilon^{1 / \alpha - \eta}, \quad \forall i \leq N^*_{ \varepsilon}} \quad = \quad 1.
\end{equation}
The advantage of dealing with the unconditioned process is that now $ N_{ \varepsilon}^*$ is distributed according to a Poisson random variable of parameter $ \Pi( \varepsilon^{ 1/\alpha}, \infty)$, that is, using \eqref{eq:levymeasure},   \begin{eqnarray} \label{eq:neps} N^*_{\varepsilon}  &  \overset{(d)}{=}&  \mathsf{Poisson}\left(   \frac{1}{\varepsilon} \cdot \frac{\alpha-1}{ \Gamma(\alpha-2)} \right) .  \end{eqnarray} 
Furthermore,  by the Markov property of the process $X$, the {random} variables $$ \mathsf{Amp}_{[t_{i}^{(\varepsilon),*},t_{i+1}^{(\varepsilon),*})} X, \qquad {i \geq 0}$$ are independent and identically distributed. By the scaling {property of $X$}, their common distribution can be written as $\varepsilon^{1/\alpha} \cdot  \mathcal{A}$, where $$ \mathcal{A} := \mathsf{Amp}_{[0,  \mathcal{E})} \tilde{X}, $$
where $ \tilde{X}$ is the LŽvy process $X$ conditioned not to make jumps larger than $1$, that is with LŽvy measure given by $ \Pi(dx) \mathbf{1}_{(0,1)}(x)$, and $\mathcal{E}$ is an independent  exponential variable of parameter $({\alpha-1})/{ \Gamma(\alpha-2)}$.

We claim that $\Es{\exp(\lambda \mathcal{A})} < \infty$ for {a certain} $ \lambda>0$. To this end, it is sufficient to check that for {a certain} $\lambda>0$ we have both $$ \Es{ \exp \left(-\lambda  \cdot \inf_{[0,\mathcal{E}]} \tilde{X} \right)} < \infty \qquad \textrm{and} \qquad \Es{ \exp \left(\lambda \cdot\sup_{[0,\mathcal{E}]} \tilde{X} \right)}< \infty.$$
The first inequality is a consequence of the discussion of \cite[p. 188]{Ber96}  applied to the spectrally negative process $-X$. For the second one, we slightly adapt these arguments: 
Since $\Delta \tilde{X}_{s}<1$ for every $s \geq 0$, by the Markov property applied at $T_{[1, \infty]} = \inf\{ t >0 : \tilde{X}_{t} \geq  1\}$ and by lack of memory of the exponential law, we have  
$$ \Pr{ \sup_{[0,\mathcal{E}]} \tilde{X} > a + 2} \leq \Pr{\sup_{[0,\mathcal{E}]} \tilde{X} > 1}\Pr{\sup_{[0,\mathcal{E}]} \tilde{X} > a}, $$ which yields $ \P( \sup_{[0,\mathcal{E}]} \tilde{X} > 2n) \leq \P(\sup_{[0,\mathcal{E}]} \tilde{X} > 1)^n$ for every $n \geq 1$. It follows that $\Es{\exp(\lambda \mathcal{A})} < \infty$ for every $ 0<\lambda<\frac{1}{2}\log \P(\sup_{[0,\mathcal{E}]} \tilde{X} > 1)$. 
To establish \eqref{eq:vers1bis}, write
 \begin{eqnarray*} && \Pr{N^*_{ \varepsilon} \geq  \varepsilon^{-1-\eta} \ \   \mbox{ or } \ \ \exists i \leq N^*_{ \varepsilon} \mbox{ s.t. } \mathsf{Amp}_{[t_{i}^{( \varepsilon),*},t_{i+1}^{( \varepsilon),*})} X \geq \varepsilon^{1 / \alpha - \eta}} \\
 && \qquad \qquad \qquad \qquad \qquad\qquad \qquad \qquad\qquad \qquad \qquad\leq \Pr{N^*_{ \varepsilon} \geq  \varepsilon^{-1-\eta}} 
  +  \varepsilon^{-1-\eta} \Pr{ \mathcal{A} \geq \varepsilon^{-\eta}}.  \end{eqnarray*}
Since $ \mathcal{A}$ has exponential moments and by \eqref{eq:neps}, the right-hand side of the last display vanishes as $ \varepsilon \to 0$. This implies \eqref{eq:vers1bis} and completes the proof of the upper bound.  \qedhere

\subsubsection{Lower bound} 

\proof Denote by $\nu$ the probability measure on $ \mathscr{L}_{\alpha}$ obtained as the push-forward of the Lebesgue measure on $[0,1]$ by the projection $ \mathbf{p}$. We will show that for every $ \delta \in (0, \alpha)$, almost surely, for $\nu$-almost every $u$ we have \begin{eqnarray} \label{eq:goalinf} \limsup_{r \to 0} \frac{ \nu( B_{r}(u))}{ r^{\alpha- \delta}} &=&0, \end{eqnarray}
where $B_{r}(u)$ is the ball of center $u$ and radius $r>0$ in the metric space $\mathscr{L}_{\alpha}$. By standard density theorems for Hausdorff measures \cite[Theorem 8.8]{Mat95} (this reference covers the case of measures on $ \R^n$, but the proof remains valid here), this implies that $ \dim_{H}( \La) \geq  \alpha-\delta$, almost surely. The lower bound will thus follow. 

Fix $ \delta  \in (0, \alpha)$. Let $U$ be a uniform variable over $[0,1]$ independent of $ \La$. We shall prove that almost surely, for every $r>0$ sufficiently small we have $\nu(B_{r}(\mathbf{p}(U))) \leq 2 r^{\alpha- \delta}$. By Fubini's theorem, this indeed implies \eqref{eq:goalinf}. We will use the following lemma:
\begin{lemma}\label{lem:borninf} Fix $\eta >0$. Almost surely, as $\varepsilon \to 0$, there exists a jump time $T_{ \varepsilon}$ of $ \X$ such that  the following three conditions hold:
\begin{enumerate}[$(i)$]
\item $T_{\varepsilon} \in (U- \varepsilon, U)$,
\item $ \min(x_{T_{ \varepsilon}}^{U},\Delta_{T_{\varepsilon}}- x_{T_{ \varepsilon}}^{U}) > \varepsilon^{1/\alpha+\eta}$,
\item $\inf_{[U, U+ \varepsilon^{1- \eta}]} X^ \mathrm{exc} < X^ \mathrm{exc}_{T_{ \varepsilon}-}$.\end{enumerate}
\end{lemma}

\begin{figure}[!h]
 \begin{center}
 \includegraphics[width=10cm]{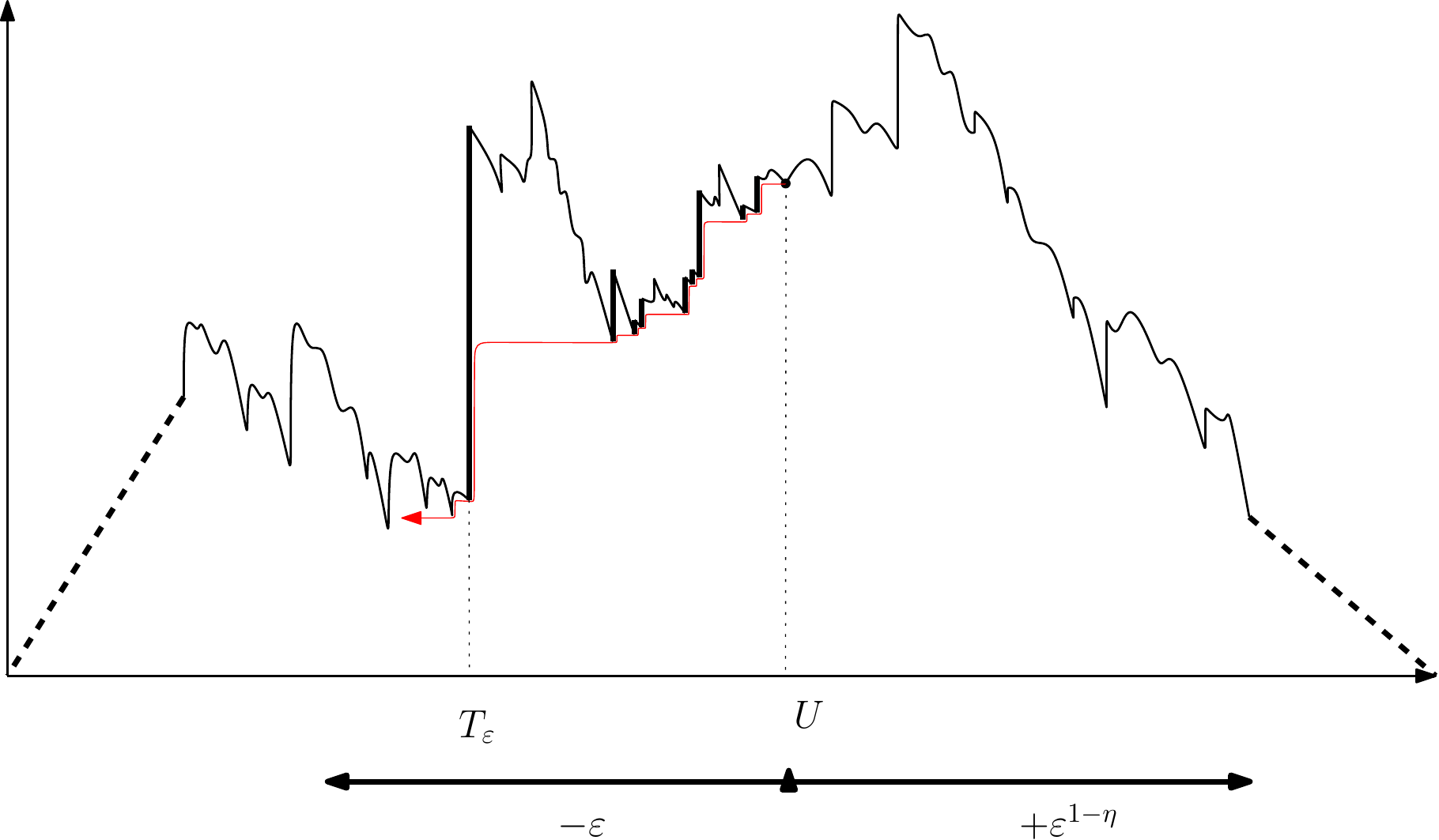}
 \caption{Setup of \cref{lem:borninf}. The red line shows the ancestral path of $U$ towards $0$ and the loops encountered during this descent.  }
 \end{center}
 \end{figure}
Assuming $(i), (ii)$ and $(iii)$,  let us show that   \begin{eqnarray*}\nu\Big(B_{ \varepsilon^{1/\alpha + \eta}}\big(\mathbf{p}(U)\big)\Big) &\leq&  2\varepsilon^{1- \eta}  \end{eqnarray*} which, together with the statement of the lemma, will imply our goal. {Indeed}, it is sufficient to check that whenever $s n [U- \varepsilon, U+ \varepsilon^{1- \eta}]$ then we have $d(s,U) \geq \varepsilon^{1/\alpha +\eta}$. {To this end, note that} if $s n [U- \varepsilon, U+ \varepsilon^{1- \eta}]$ then $(iii)$ and $(i)$ show that $s \wedge U < T_{ \varepsilon}$ and hence $ s \wedge U \prec T_{ \varepsilon} \prec U$.  By the definition of $d$ and \cref{lem:controls} $(i)$ we get 
$$ d(s,U) \geq \min ( x_{T_{ \varepsilon}^{U}} , \Delta_{T_{ \varepsilon}}-x_{T_{ \varepsilon}^{U}}) \geq \varepsilon ^{1/\alpha + \eta},$$ 
as desired.

It thus remains to show \cref{lem:borninf}. Since the statement we intend to prove is a local statement around the point $U$ in $ X^{ \mathrm{exc}}$,  by standard arguments involving continuity relations between $X$  and the LŽvy bridge $ X^ \br$ as well as the Vervaat transformation between $ X^ \br$ and $X^{ \mathrm{exc}}$ (see \cref{sec:absolutecontinuity}) it suffices to prove \cref{lem:borninf} when $X^ \mathrm{exc}$ is replaced by {a two-sided LŽvy process $(X_{t})_{t \in \R}$} and the point $U$ by  the point $0$. Recall from the statement of \cref{cor:technique} the definition of the event
 $$A_{ \varepsilon} = \left\{ \exists s \in [-\varepsilon,0] \mbox{ with } s \preccurlyeq 0 :  x_{s}^{0}(X) \geq  \varepsilon^{1/\alpha +\eta}\ \  \mbox{ and } \ \ \Delta X_{s} - x_{s}^{0}(X) \geq \varepsilon^{1/\alpha +\eta}\right\}.$$
By \cref{cor:technique}, there exist $C, \gamma>0$ such that $ \Pr{A_{ \epsilon}^c}< C \epsilon^{ \gamma}$. Borel--Cantelli's Lemma implies that a.s.\, $A_{2^{-k}}$ holds for every $k$ sufficiently large. This proves $(i)$ and $(ii)$ (with a slightly larger $\eta$). Next, by \cite[Chapter VIII, Theorem 6 (i)]{Ber96}, a.s.\,there exists $c>0$ such that for every $ \epsilon$ sufficiently small $\sup_{[0, \varepsilon^{1- \eta}]} (-X) \geq	 c \epsilon^{(1- \eta/2)/ \alpha}$, and by the last line of the proof of  Theorem 5 in \cite[Chapter VIII]{Ber96},  a.s.\,there exists $C>0$ such that for every $ \epsilon$ sufficiently small, $ \sup_{[0, \varepsilon]}(-X) \leq C \epsilon^{(1- \eta/3)/ \alpha}$. It follows that a.s.\,for every $ \epsilon$ sufficiently small we have $$\inf_{[0, \varepsilon^{1- \eta}]} X < \inf_{[- \varepsilon,0]}X.$$
Combined with $(i)$, this implies $(iii)$ and completes the proof.\endproof

 \section{Invariance principles for discrete looptrees}
 
\subsection{Plane trees and Lukasiewicz path} \label{sec:planetrees}
We briefly recall the
formalism of plane trees, which can for instance be found in \cite{Nev86,LG05}. Let $\N=\{0,1,\ldots\}$ be the set of nonnegative integers,
$\N^*=\{1,\ldots\}$ and let $ \mathcal{U}$ be the set of labels
$$ \mathcal{U}=\bigcup_{n=0}^{\infty} (\N^*)^n,$$
where by convention $(\N^*)^0=\{\varnothing\}$. An element of $ \mathcal{U}$ is
a sequence $u=u_1 \cdots u_m$ of positive integers, and we set
$|u|=m$, which represents the ``generation'' or heightof $u$. If $u=u_1
\cdots u_m$ and $v=v_1 \cdots v_n$ belong to $ \mathcal{U}$, we write $uv=u_1
\cdots u_m v_1 \cdots v_n$ for the concatenation of $u$ and $v$. Finally, a \emph{plane tree} $\tau$ is a finite  subset of $
\mathcal{U}$ such that:
\begin{itemize}
\item[1.] $\varnothing \in \tau$,
\item[2.] if $v \in \tau$ and $v=uj$ for some $j \in \N^*$, then $u
\in \tau$,
\item[3.] for every $u \in \tau$, there exists an integer $k_u(\tau)
\geq 0$ (the number of children of $u$) such that, for every $j \in \N^*$, $uj \in \tau$ if and only
if $1 \leq j \leq k_u(\tau)$.
\end{itemize}
In the following, by \emph{tree} we will always mean plane tree. We denote the set of all trees by $\T$. 
We will often view each vertex of a tree $\tau$ as an individual of
a population whose $\tau$ is the genealogical tree. If $u, v \in \tau$ we denote by $\llbracket u,v \rrbracket$ the discrete geodesic path between $u$ and $b$ in $\tau$.
The total
progeny of $\tau$, which is the total number of vertices of $\tau$,
will be
denoted by $|\tau|$. The number of leaves (vertices $u$ of $ \tau$ such that $k_u( \tau)=0$) of the tree $\tau$ is denoted by $\lambda( \tau)$ and the height of the tree (which is the maximal generation) is denoted by $ \mathsf{H}( \tau_{n})$.  

\bigskip

We now recall the classical coding of plane trees by the so-called
Lukasiewicz path. This coding is crucial in the understanding of
scaling limits of discrete looptrees associated with large trees. Let $\tau$ be a plane tree whose vertices are listed in lexicographical order $\varnothing=u(0)<u(1)<\cdots<u(|\tau|-1)$.

\begin{figure}[h!]
\begin{center}
\includegraphics[width=12cm]{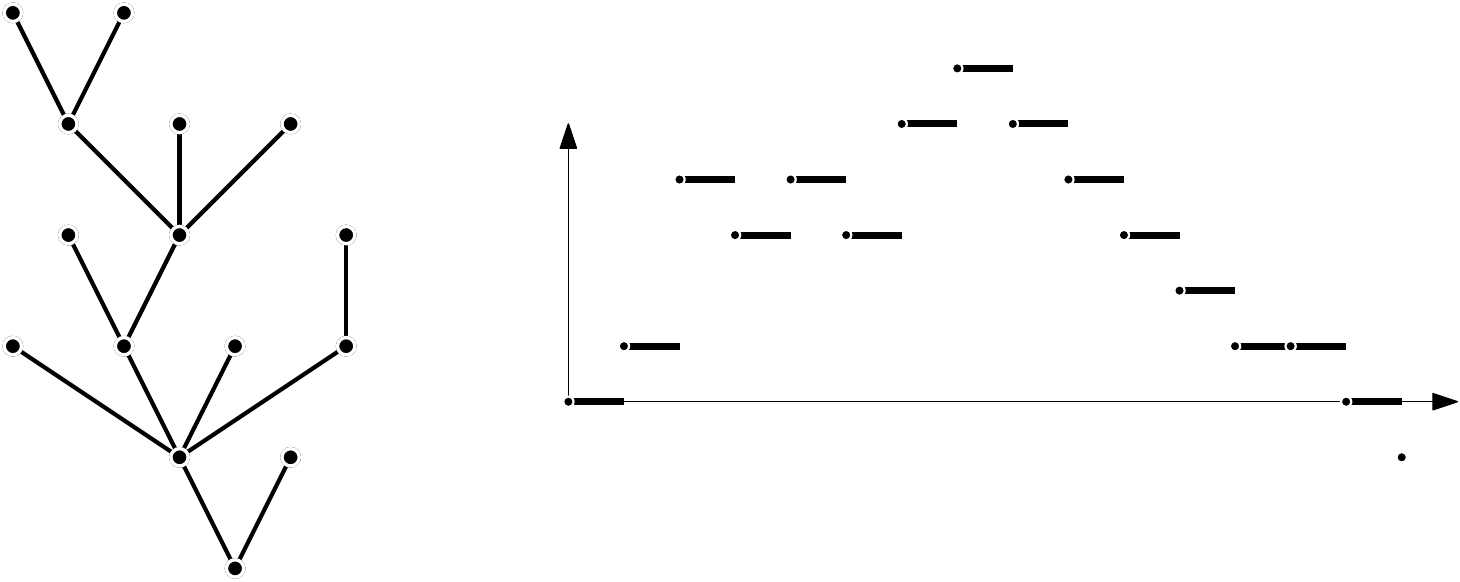}
   \caption{\label{fig:luka} A tree and its Lukasiewicz path.}
\end{center}
\end{figure}
The Lukasiewicz path $ \W(\tau)=(\W_n(\tau), 0 \leq n \leq |\tau|)$ is defined by
$\W_0(\tau)=0$ and $\W_{n+1}(\tau)=\W_{n}(\tau)+k_{u(n)}(\tau)-1$ for $0 \leq n \leq |\tau|-1$  (see \cref{fig:luka} for an example, where $W$ is interpolated into a cˆdlˆg function between successive integers). It is easy to see that $\W_n ( \tau) \geq 0$ for $0 \leq n < |\tau|$ but $\W_{|\tau|}(\tau)=-1$ (see e.g. \cite[Proposition 1.1]{LG05}).

\subsection{Invariance principles for discrete looptrees}

Recall from the Introduction that a discrete looptree  $ \mathsf{Loop}( \tau)$ is associated with every plane tree $ \tau \neq \varnothing$ (see \cref{fig:loop}). In this section, we give a sufficient condition on a sequence of trees $( \tau_{n})_{n \geq 1}$ that ensures that the associated looptrees $ ( \mathsf{Loop}( \tau_{n}))_{n \geq 1}$, appropriatly rescaled, converge towards the stable looptree $ \La$.
  
\begin{theorem}[Invariance principle] \label{thm:continuity} Let $(\tau_{n})_{n \geq 1}$ be a sequence of random trees such that there exists a sequence $(B_{n})_{n\geq0}$ of positive real numbers  
satisfying 
$$ (i) \quad \left(\frac{1}{B_{n}}\W_{\lfloor  |\tau_{n}|t \rfloor}(\tau_{n}); \, 0 \leq t \leq 1 \right)  \quad\mathop{\longrightarrow}^{(d)}_{n \rightarrow \infty} \quad  X^{ \mathrm{exc}, ( \alpha)}, \qquad (ii)\quad\frac{1}{B_{n}} \mathsf{H}(\tau_{n})   \quad\mathop{\longrightarrow}^{( \P)}_{n \rightarrow \infty} \quad  0,$$
where the first convergence holds in  distribution for the Skorokhod topology on $ \D([0,1], \R)$ and the second convergence holds in probability.
Then the  convergence
 $$ \displaystyle \frac{1}{B_{n}} \cdot \mathsf{Loop}( \tau_{n})  \quad \xrightarrow[n\to\infty]{(d)} \quad  \mathscr{L}_{ \alpha}$$
holds in distribution for the Gromov--Hausdorff topology.
\end{theorem}

Of course, the main applications of this {result} concern Galton--Watson trees. If $\rho$ is a probability measure on $\N$ such that $\rho(1)<1$, we denote by $ \mathsf{GW}_\rho$ the law of a Galton--Watson tree with offspring distribution $ \rho$. We say that $ \rho$ is critical if it has mean equal to $1$.

If $\rho$ is a critical offspring distribution in the domain of
attraction of a stable law\footnote{Recall that this
means that $\mu([j,\infty))= j^{- \alpha} L(j)$,
where $L: \R_+ \rightarrow \R_+$ is a function such that $L(x)>0$ for $x$ large enough and $\lim_{x
\rightarrow \infty} L(tx)/L(x)=1$ for all $t>0$ (such a function is
called slowly varying). We refer to \cite{BGT89} for details.} of index $ \alpha \in (1,2)$, Duquesne \cite{Du03} showed that $ \GW_{ \rho}$ trees conditioned to have $n$ vertices (provided this conditioning makes sense) satisfy the assumptions of \cref{thm:continuity} ((i) follows from Proposition 4.3 and the proof of Theorem 3.1 in \cite{Du03}, and (ii) follows from the fact that $H( \tau_{n}) \cdot B_{n}/n$ converges in distribution to a positive real valued random variable as $n \rightarrow \infty$ by \cite[Theorem 3.1]{Du03}). Recently, the second author \cite{Kor12} proved the same result  for $ \GW_{ \rho}$ trees conditioned to have $n$ leaves. 

\begin{remark} Let us mention that a different phenomenon happens when the offspring distribution $ \rho$ is critical and has finite variance: in this case, if $\tau_{n}$ denotes a $\GW_{ \rho}$ tree conditioned to have $n$ vertices, it is shown in \cite {CHK13+} that $  \mathsf{Loop}( \tau_{n})/ \sqrt {n}$ converges in distribution towards a constant times the Brownian CRT, and the constant depends this time on the offspring distribution in a rather complicated fashion (in \cite {CHK13+}  this is actually established under the condition that  $ \rho$ has a finite  exponential moment). The main difference is that in the finite variance case, $B_{n}$ is a constant times $ \sqrt {n}$, and $ H( \tau_{n})/B_{n}$ does not converge in probability to $0$ any more, but converges in distribution to a positive real-valued random variable. \end{remark}

\begin{remark} Condition $(ii)$ of the above theorem ensures that the height of $\tau_{n}$ is negligeable compared to the typical size of loops in $  \mathsf{Loop}( \tau_{n})$, so that asymptotically distances in $ \tau_{n}$ do not contribute to the distances in  $  \mathsf{Loop}( \tau_{n})$. Also observe that, in the boundary case $\alpha=2$, when $\rho$ has infinite variance (so that $ \rho$ is in the domain of attraction of the Gaussian law), we still have $H( \tau_{n})/ B_{n} \to 0$ (by the same argument that follows \eqref{eq:cvleaves}). In analogy with Theorem \ref{thm:1and2} (ii) we believe that, in this case,  $B_{n}^{-1} \cdot \mathsf{Loop}( \tau_{n})$ converges in distribution as $n \rightarrow \infty$ towards $ \frac{1}{2} \cdot \mathcal{T}_{2}$. \end{remark}

An immediate corollary of \cref{thm:continuity} is that $ \La$ is a length space (see \cite[Chapter 2]{BBI01} for the definition of a length space):

\begin{corollary}\label{cor:lengthspace}Almost surely, $ \La$ is a length space.
\end{corollary}
\proof This is a consequence of \cite[Theorem 7.5.1]{BBI01}, since by \cref{thm:continuity}, the space $ \La$ is  a Gromov--Hausdorff limit of finite metric spaces. \endproof

\proof[Proof of \cref{thm:continuity}] Let $(\tau_{n})_{n \geq 1}$ be a sequence of random trees and $(B_{n})_{n \geq 1}$ a sequence satisfying the assumptions $(i)$ and $(ii)$. {Note that necessarily $B_{n} \rightarrow \infty$ as $n \rightarrow \infty$.} The Skorokhod representation theorem allows us to assume that the convergences $(i)$ and $(ii)$ hold almost surely and we aim at proving an almost sure convergence of  $ B_{n}^{-1} \cdot \mathsf{Loop}( \tau_{n})$ towards $ \La$.  We first define a sequence of finite metric spaces denoted by  $\mathsf{Loop'}( \tau_{n})$ which are slightly different from $\mathsf{Loop}( \tau_{n})$, but more convenient to work with. Let $u_{0}^{n}, u_{1}^{n}, \ldots, u_{|\tau_{n}|-1}^{n}$ be the vertices of $ \tau_{n}$ listed in lexicographical order, then ${\mathsf{Loop'}}( \tau_{n})$ is by definition the graph on the set of vertices of $\tau_{n}$ such that two vertices $u$ and $v$ are joined by an edge if and only if one of the following three conditions are satisfied in $ \tau$: $u$ and $v$ are consecutive siblings of a same parent, or $u$ is the first sibling (in the lexicographical order) of $v$, or $u$ is the last sibling of $v$. In particular, if $u$ has a unique child $v$ in $ \tau$, then $u$ and $v$ are joined by two edges in ${\mathsf{Loop'}}( \tau_{n})$. See \cref{fig:loop_tilde} for an example. We equip ${\mathsf{Loop'}}( \tau_{n})$ with the graph metric.

\begin{figure}[!h]
 \begin{center}
 \includegraphics[width=0.75 \linewidth]{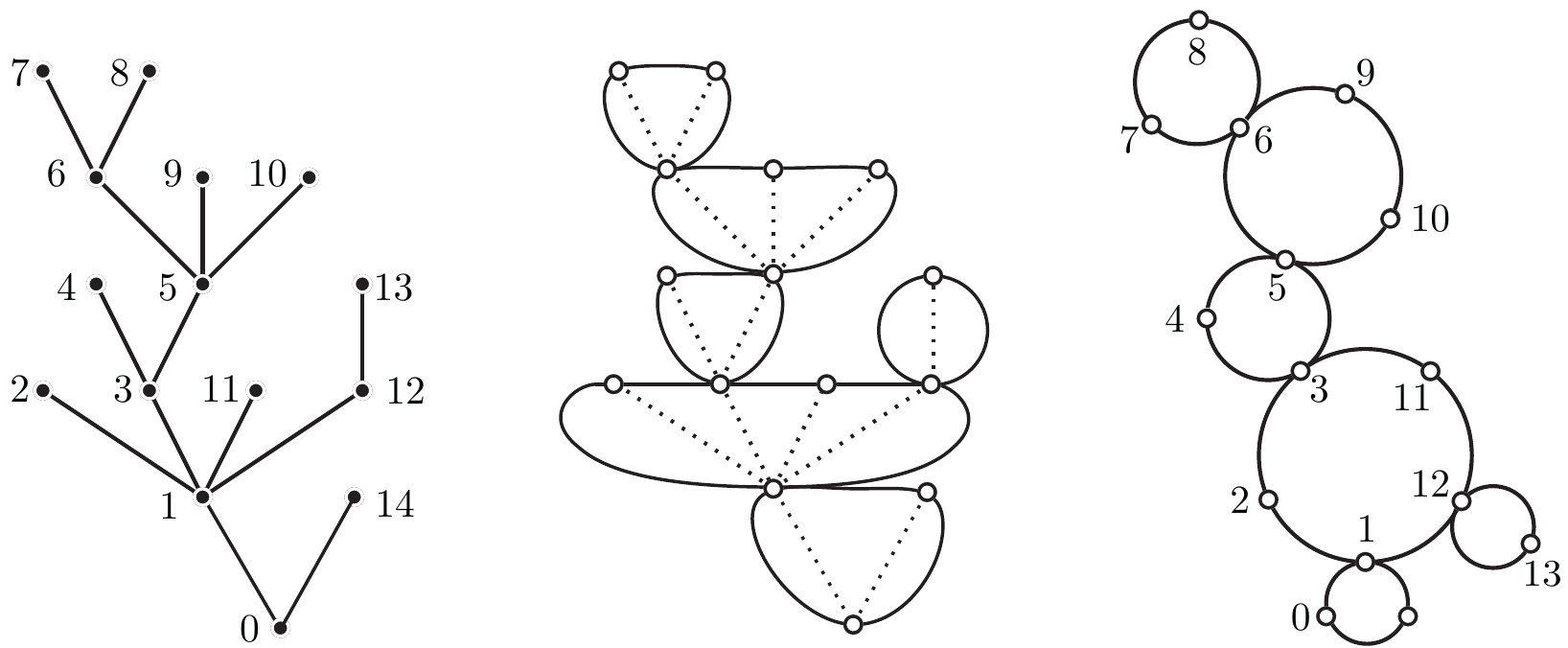}
 \caption{ \label{fig:loop_tilde}A discrete tree $\tau$ and  $ {\mathsf{Loop'}}( \tau)$.}
 \end{center}
 \end{figure}
 
It is easy to check that ${\mathsf{Loop'}}( \tau_{n})$ is at Gromov--Hausdorff distance at most $2$ from $ \mathsf{Loop}(\tau_{n})$ (compare \cref{fig:loop,fig:loop_tilde}). Since $ B_{n} \to \infty$ as $n \to \infty$, it is thus sufficient to show that
 \begin{equation}
 \label{eq:amq} \frac{1}{B_{n}} {\mathsf{Loop'}}( \tau_{n})  \quad \xrightarrow[n\to\infty]{a.s.} \quad  \mathscr{L}_{ \alpha}.
 \end{equation}
Recall that $ \mathbf{p} : [0,1] \to \mathscr{L}_{\alpha}$ denotes the canonical projection. For every $n \geq 1$,  we let  $ \mathcal{R}_{n}$ be the correspondence between $ \mathscr{L}_{\alpha}$ and $ B_{n}^{-1} \cdot \mathsf{Loop'}( \tau_{n})$  made of all the pairs $( \mathbf{p}(s),u_{i}^{n})$ such that $ i = \lfloor |\tau_{n}|s \rfloor  \pm 1$ where $s \in [0,1]$ and $i \in \{0,1,2, \ldots , |\tau_{n}|-1\}$. It is easy to check that $ \mathcal{R}_{n}$ is indeed a correspondence and we will show that, under our assumptions, its distortion vanishes as $n \to \infty$. 

To do so, we shall first see that the graph distance $  d'_{n}$ {of} $ \mathsf{Loop'}(\tau_{n})$  can be expressed in a very similar way to \eqref{eq:def2}.  To simplify notation, we denote by $( \W_{k}^n)_{0 \leq k \leq |\tau_{n}|}$ the Lukasiewicz path associated with $\tau_{n}$. By definition of $\W^n$, the vertex $u_{i}^n$ has $$ \Delta \W_{i}^n \quad := \quad\W_{i+1}^n-\W_{i}^n+1$$ children.  In addition, the discrete genealogical order (also denoted by $\preccurlyeq$) on $u_{0}^{n},  \ldots , u_{|\tau_{n}|-1}^{n}$ can be recovered from $\W^n$ in a similar way to the continuous setting (see the proof of Proposition 1.2 in \cite {LG05} for details): 
$$u_{i}^{n} \preccurlyeq u_{j}^{n} \qquad \textrm{if and only if}  \qquad i \leq j \textrm{ and }\inf_{i \leq m \leq j}\W_{m}^n = \W_{i}^n.$$ Furthermore, when $u_{i}^n \prec u_{j}^n$, that is when $ u_{i}^n \preccurlyeq u_{j}^n$ and $i \ne j$, the quantity 
$$ x^ {j}_ {n,i} \quad := \quad \inf_{i +1 \leq k \leq j }\W_{k}^n - \W_{i}^n + 1$$
informally gives  the ``position'' of the ancestral line of $u_{j}^{n}$ with respect to $u_{i}^{n}$; more precisely the  $ (\Delta \W_{i}^n-x_{n,i}^{j} + 1)$-th child of $u_{i}^n$ (in the lexicographical order) is an ancestor of $u_{j}^{n}$. 
Similarly to the continuous setting, one checks that the distance between $u_{i}^n \preccurlyeq u_{j}^n$ in $ \mathsf{Loop'}(\tau_{n})$ is given by 
\begin{equation}
\label{eq:ddiscret}{d'_{n}}(u_{i}^{n},u_{j}^{n}) = \sum_{ u^{n}_{i} \preccurlyeq u^{n}_{k} \prec u^{n}_{j}}\delta_{{n,k}}( 0, {x_{n,k }^ {j}}),
\end{equation}
where by definition $\delta_{n,i}(a,b) = |b-a| \wedge ( \Delta \W_{i}^n+1-|b-a|)$ for $a,b \in \{0, 1, 2, \ldots , \Delta \W_{i}^n\}$. If $u_{i}^n$ is not an ancestor of $u_{j}^n$, then the distance between  $u_{i}^n$ and  $u_{j}^n$ in $ \mathsf{Loop'}(\tau_{n})$ can be computed by breaking in three parts the geodesic between $u_{i}^n$ and  $u_{j}^n$ at their most recent common ancestor as in the continuous case (see \eqref{eq:def2}): if $u_{m}^n$ is the most recent common ancestor of $u_{i}^n$ and  $u_{j}^n$, then
\begin{equation}
\label{eq:def2discret}{d'_{n}}(u_{i}^{n},u_{j}^{n})= \delta_{n,m}(x_{n,m}^i,x_{n,m}^j)+ \sum_{ u^{n}_{m} \prec u^{n}_{k} \prec u^{n}_{i}}\delta_{{n,k}}( 0, {x_{n,k }^ {i}})+\sum_{ u^{n}_{m} \prec u^{n}_{k} \prec u^{n}_{j}}\delta_{{n,k}}( 0, {x_{n,k }^ {j}}).
\end{equation}

Now, we argue by contradiction and suppose that there exists $\varepsilon >0$, $i_{n},j_{n} \in \{0,1,  \ldots , |\tau_{n}|-1\}$, $ s_{n}, t_{n}\in [0,1]$ such that $(u^n_{i_{n}},  \pi(s_{n})) \in  \mathcal{R}_{n}$ and $(u_{j_{n}}^n , \pi(t_{n})) \in  \mathcal{R}_{n}$, and such that for every $n$ sufficiently large \begin{eqnarray} \label{eq:faux} \left| \frac{1}{B_{n}}d_{n}'(u_{i_{n}}^n,u_{j_{n}}^n) - d(s_{n},t_{n}) \right| \geq \varepsilon. \end{eqnarray}
By compactness, we may assume without loss of generality that $i_{n}/|\tau_{n}| \to s$ and $j_{n} / |\tau_{n}| \to t$. Because $i_{n} = \lfloor s_{n}|\tau_{n}|\rfloor \pm 1$, we also have $s_{n} \to s$ and similarly $t_{n}\to t$. We make the additional assumption that $u_{i_{n}}^n \preccurlyeq u_{j_{n}}^n$ and $s_{n} \preccurlyeq t_{n}$ for every $n$ sufficiently large. Note that this entails $ s \preccurlyeq t$. The general case is more tedious and can be solved by breaking at the most recent common ancestor  and using \eqref{eq:def2discret} instead of \eqref{eq:ddiscret}. We leave details to the reader.

The idea is now clear: On the one hand,  jumps of $\W^n$ converge after scaling towards the jumps of $X^{ \mathrm{exc}}$ and on the other hand $d$ and $d_{n}'$ have similar expressions involving their jumps (compare \eqref{eq:def1} and \eqref{eq:ddiscret}). Thus, intuitively, the inequality \eqref{eq:faux} cannot hold for $n$ sufficiently large. Let us prove this carefully. Since $  \{r \in [0,1]; \, s \preccurlyeq r \prec t \textrm{ and } \Delta_{r}>0\}$ is countable, by \eqref{eq:def1} there exists $\eta >0$ such that 
\begin{equation}
\label{eq:eps2}\sum_{\begin{subarray}{c} s \preccurlyeq r \prec t \\ \delta_{r}(0,x_{r}^t) > \eta \end{subarray}} \delta_{r}(0,x_{r}^t) \geq d(s,t) - \frac{ \varepsilon}{4}.
\end{equation}
Note that the sum appearing in the last expression contains a finite number of terms. To simplify notation, write $\{r \in [0,1]; \, s \preccurlyeq r \prec t \textrm{ and } \delta_{r}(0,x_{r}^t) > \eta \}=  \{r_{0},r_{1}, \ldots, r_{m}\}$ with $r_{0} \prec r_{1} \prec r_{2} \prec \ldots \prec r_{m} < t$ and possibly $r_{0}=s$. We shall now show that
\begin{equation}
\label{eq:s1} \sum_{\begin{subarray}{c} s \preccurlyeq r \prec t \\ \delta_{r}(0,x_{r}^t) > \eta \end{subarray}} \delta_{r}(0,x_{r}^t)  -  \frac{1}{B_{n}} \sum_{u_{i_{n}}^n \preccurlyeq u_{k}^n \prec u_{j_{n}}^n}\delta_{n,k}(0, x_{n,k}^{j_{n}}) \mathbbm{1}_{ \begin{subarray}{l}\delta_{n,k}(0, x_{n,k}^{j_{n}}) > \eta\cdot B_{n} \end{subarray}}  \quad\mathop{\longrightarrow}_{n \rightarrow \infty} \quad 0.
\end{equation}
%We first mention that some care is needed to establish \eqref{eq:s1} when $ \Delta_{t}>0$. 
 {Properties of the Skorokhod topology entail that the jumps of $\W^n/B_{n}$ converge towards the jumps of $ X^ \mathrm{exc}$, together with their locations.} It follows that for every $r \in \{r_{0}, r_{1}, r_{2}, \ldots ,r_{m}\}$ one can find $k_{n}(r) \in \{0,1,  \ldots , |\tau_{n}|-1\}$  such that the following two conditions hold for $n$ sufficiently large (see \cref{fig:explication} for an illustration):

\medskip

$(i)$ $ \displaystyle \frac{k_{n}(r)}{|\tau_{n}|} \to r, \quad u_{i_{n}}^n \preccurlyeq u_{k_{n}(r)}^n \preccurlyeq u_{j_{n}}^n, \quad \frac{1}{B_{n}}\delta_{n,k_{n}(r)}(0,x_{n,k_{n}(r)}^{j_{n}}) \to \delta_{r}(0, x_{r}^t)\quad \mbox{ as }n \to \infty,$

\medskip

$(ii)$ 
$  \displaystyle \left\{k_{n}(r_{0}), \ldots, k_{n}(r_{m}) \right\} = \left\{k_{n}; \, i_{n} \preccurlyeq k_{n} \preccurlyeq j_{n} \textrm{ such that } \delta_{n,k}(0, x_{n,k}^{j_{n}}) > \eta \cdot B_{n} \right\}.$

\noindent This implies \eqref{eq:s1}. { In $(i)$, when $r=r_{0}$, we use the fact that $ s_{n} \preccurlyeq t_{n}$ for every $n \geq 1$. 

\begin{figure}[!h]
 \begin{center}
 \includegraphics[width=0.9 \linewidth]{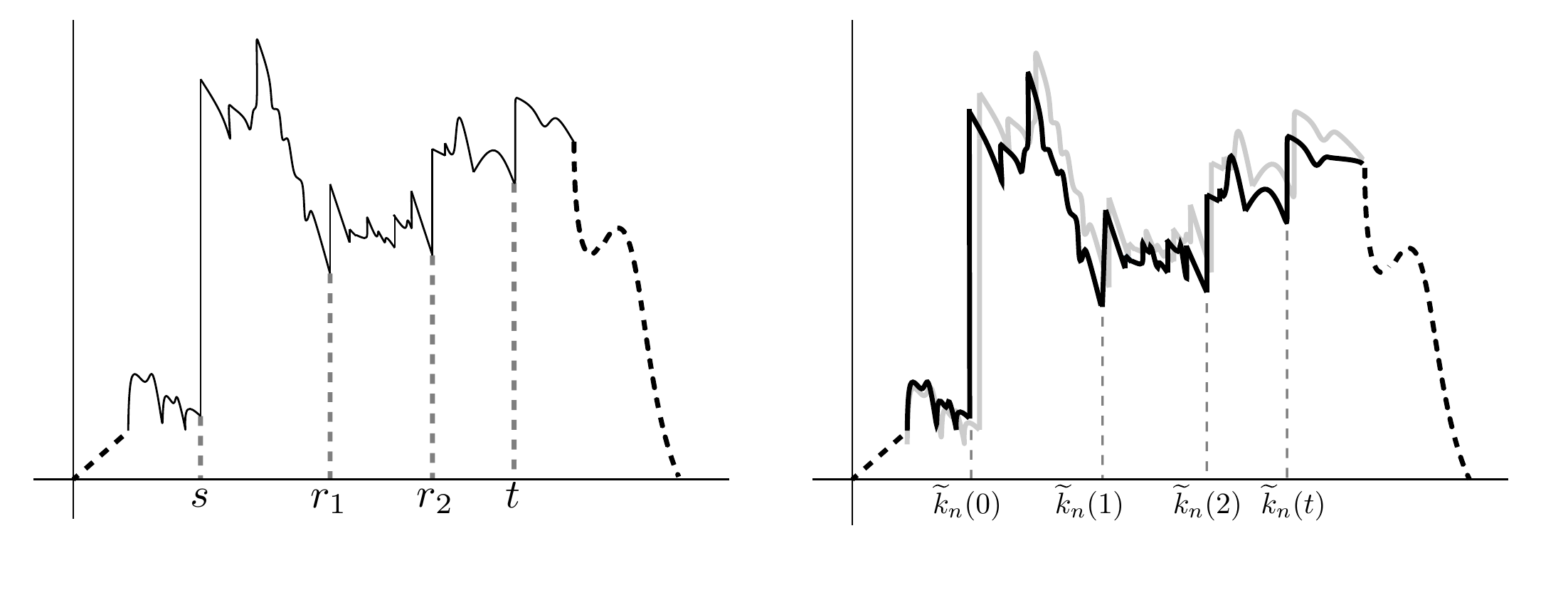}
 \caption{ \label{fig:explication}Illustration of the conditions (i) and (ii) above. In the figure in the right, the black process is $W^n/B_{n}$ and the  grey one is $ \X$. To simplify, here we have set $  \widetilde {k}_{n}(i)= k_{n}(r_{i})/ | \tau_{n}|$ and $  \widetilde {k}_{n}(t)= k_{n}(t)/ | \tau_{n}|$.}
 \end{center}
 \end{figure}

By combining \eqref{eq:eps2} and \eqref{eq:s1}, we get that
 \begin{eqnarray*} \limsup_{n \to \infty} \left| d(s,t)-d_{n}'(i_{n},j_{n})\right| &=& \limsup_{n\to \infty} \left| \sum_{s \preccurlyeq r \prec t}  \delta_{r}(0,x_{r}^t) -  \frac{1}{B_{n}}\sum_{u_{i_{n}}^n \preccurlyeq u_{k}^n \prec u_{j_{n}}^n} \delta_{n,k}(0, x_{n,k}^{j_{n}})\right| \\ 
 & \leq & \frac{ \varepsilon}{4} + \limsup_{n \to \infty} \frac{1}{B_{n}} \sum_{u_{i_{n}}^n \preccurlyeq u_{k}^n \prec u_{j_{n}}^n}\delta_{n,k}(0, x_{n,k}^{j_{n}}) \mathbbm{1}_{ x_{n,k}^{j_{n}} \leq \eta\cdot B_{n}} \end{eqnarray*}
 In order to get the desired contradiction, we show that the second term in the last display can be made less than $ \varepsilon/4$ provided that $ \eta >0$ is small enough. Indeed, we have
 \begin{eqnarray}
 \sum_{u_{i_{n}}^n \preccurlyeq u_{k}^n \prec u_{j_{n}}^n}\delta_{n,k}(0, x_{n,k}^{j_{n}}) \mathbbm{1}_{ x_{n,k}^{j_{n}} \leq \eta\cdot B_{n}}  &\leq& \sum_{u_{k}^n \prec u_{j_{n}}^n} x_{n,k}^{j_{n}} \mathbbm{1}_{ x_{n,k}^{j_{n}} \leq \eta\cdot B_{n}}  \notag\\
 &=&\sum_{u_{k}^n \prec u_{j_{n}}^n} x_{n,k}^{j_{n}} - \sum_{u_{k}^n \prec u_{j_{n}}^n} x_{n,k}^{j_{n}} \mathbbm{1}_{ x_{n,k}^{j_{n}} > \eta\cdot B_{n}}. \label{eq:eg}
 \end{eqnarray} The following equality will be useful 
\begin{equation}
\label{eq:use}\sum_{u_{k}^n \prec u_{j_{n}}^n}x_{n,k}^{j_{n}} = \mathrm{Height} (u_{j_{n}}^n) + \W^n_{j_{n}}.
\end{equation}
Since $ j_{n}/n \rightarrow t$, it sufficient to treat the case where either $\W^n_{j_{n}}/B_{n} \rightarrow \X_{t}$ or $\W^n_{j_{n}}/B_{n} \rightarrow \X_{t-}$. We first suppose that $\W^n_{j_{n}}/B_{n} \rightarrow \X_{t}$. At this point, we crucially use Corollary \ref{coro:exc} and assume that $\eta >0$ has been chosen  sufficiently small such that $$ \sum_{r \preccurlyeq t} x_{r}^t  \mathbbm{1}_{x_{r}^t > \eta'} \geq X_{t}^ \mathrm{exc} - \varepsilon/4.$$
 The same argument that led us to \eqref{eq:s1} entails  $$\sum_{u_{k}^n \prec u_{j_{n}}^n} x_{n,k}^{j_{n}} \mathbbm{1}_{ x_{n,k}^{j_{n}} > \eta'\cdot B_{n}}  \quad\mathop{\longrightarrow}_{n \rightarrow \infty} \quad \sum_{ r \preccurlyeq t} x_{r}^t \mathbbm{1}_{x_{r}^t > \eta'}.$$ 
{Note that we have used the fact that $\W^n_{j_{n}}/B_{n} \rightarrow \X_{t}$  in order to capture the term of the right-hand side corresponding to $r=t$.}  Consequently, combining the last display with \eqref{eq:use} and Assumption $(ii)$ of the theorem, we deduce that \eqref{eq:eg} becomes for every $n$ sufficiently large
$$\sum_{u_{i_{n}}^n \preccurlyeq u_{k}^n \prec u_{j_{n}}^n}\delta_{n,k}(0, x_{n,k}^{j_{n}}) \mathbbm{1}_{ x_{n,k}^{j_{n}} \leq \eta\cdot B_{n}} \leq  \varepsilon/4.$$
In the case $\W^n_{j_{n}}/B_{n} \rightarrow \X_{t-}$, the same argument applies after replacing every occurrence of $X_{t}^ \mathrm{exc}$ by $X_{t-}^ \mathrm{exc}$ and every occurrence of $r \preccurlyeq t$ by $r \prec t$. This completes the proof of the claim and  of \cref{thm:continuity}. \endproof

\subsection{Application to scaling limit of discrete non-crossing configurations}

We now give an application of the invariance principle established in the previous section by showing that  stable looptrees appear as Gromov--Hausdorff  limits of random Boltzmann dissections of  \cite{Kor11}. \bigskip

For every integer $n \geq 3$, recall from the Introduction that a dissection of the regular polygon  $P_{n}$ 
is the union of the sides of $P_n$ and of a collection of diagonals that may intersect only at
their endpoints, see \cref{fig:dual}. The faces are the connected components of the complement of the dissection in the polygon.

Recall from the Introduction the Boltzmann
probability measure $\mathbb{P} ^ { \mu}_ {n}$ on $ \mathbf{D}_n$,  the set of all
dissections of $P_{n+1}$. Our goal is to study  scaling limits of  random dissections $ \mathcal{D}^\mu_{n}$ sampled according to $  \mathbb{P}^\mu_{n}$ and prove \cref{cor:discretencstable}. Recall that $  \mathcal{D}^\mu_{n}$ is viewed as a metric space by endowing  the vertices of $  \mathcal{D}^\mu_{n}$ with the graph distance.
\paragraph{Duality with trees.}
The main tool is to use a bijection with trees. Indeed, the dual tree of $ \mathcal{D}^\mu_{n}$ is a Galton--Watson tree as we now explain.

Given a dissection $ \mathcal{D} \in \mathbf{D}_n$, we construct a (rooted ordered) tree $\phi( \mathcal{D})$ as follows: Consider the ``dual" graph
of $ \mathcal{D}$, obtained by placing a vertex inside each face of
$ \mathcal{D}$ and outside each side of the polygon $P_{n+1}$ and by
joining two vertices if the corresponding faces share a common edge,
thus giving a connected graph without cycles. Then remove the dual
edge intersecting the side of $P_{n+1}$ which connects $1$ to
$e^{\frac{2 \textrm{i} \pi}{n+1}}$. Finally, root the tree at the corner
adjacent to the latter side (see \cref{fig:dual}).

\begin{figure*}[!h]
\begin{center}
\includegraphics[scale=0.6]{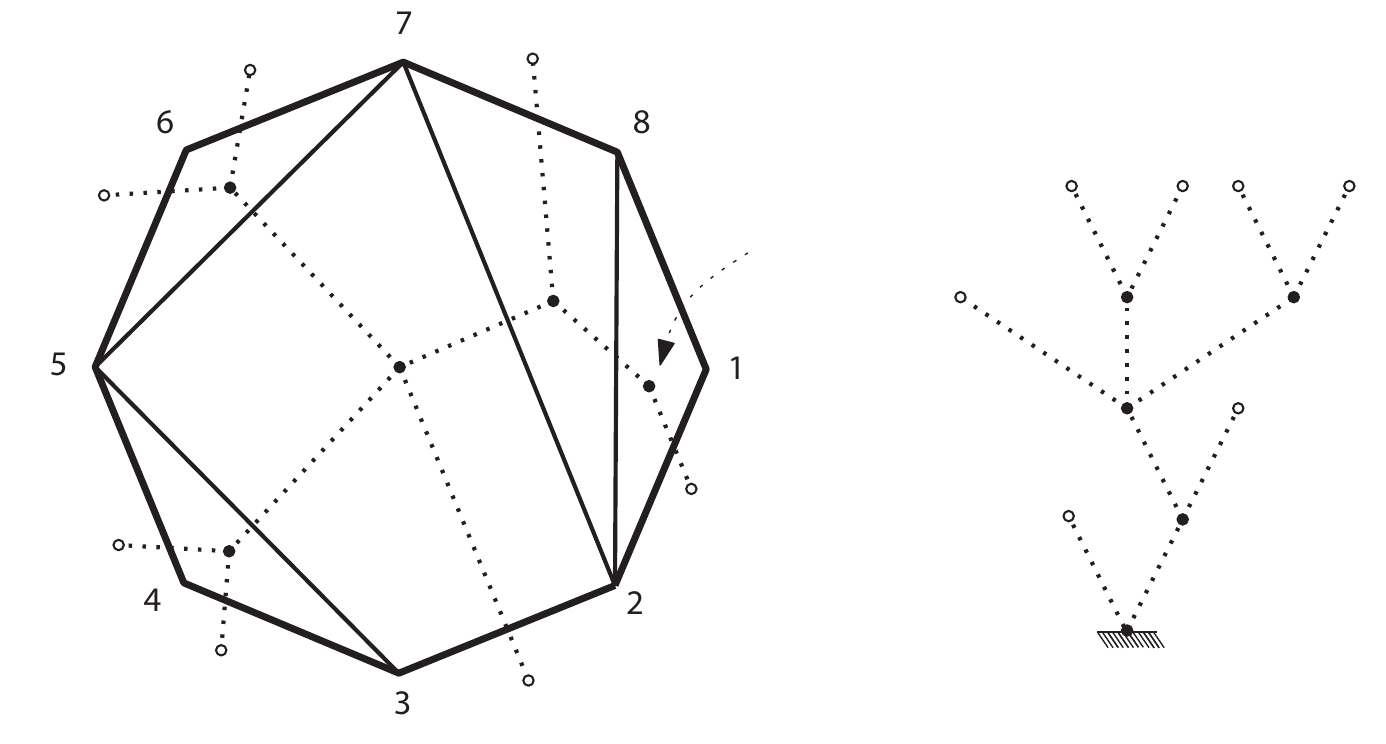}
\caption{\label{fig:dual} The dual tree of a dissection of $P_{8}$, note that the tree has $7$ leaves.}
\end{center}
\end{figure*}

We denote by $ \T^{(\ell)}_n$ the set of all plane trees with $n$ leaves such that there is no vertex with exactly one child. It is plain that the dual tree of a dissection is such a tree and the duality application $\phi$ is a bijection between $ \mathbf{D}_{n}$ and $ \T^{(\ell)}_n$. Finally, recall that $ \lambda(\tau)$ is the number of leaves of a tree $ \tau$.
The following proposition is \cite[Proposition 1.4]{Kor11}. 

\begin{proposition}\label{prop:GW} Let $\mu$ be a probability distribution over $\{0,2,3,4 \ldots\}$ of mean $1$. For every $n$ such that $ \mathsf{GW}_{\mu}( \lambda(\tau)=n)>0$, the dual tree $ \phi( \mathcal{D}_{n}^\mu)$ of a random dissection distributed according to $ \mathbb{P}_{n}^\mu$ is distributed according to $ \mathsf{GW}_{\mu}( . \mid \lambda(\tau)=n).$
\end{proposition}

With all the tools that we have in our hands, the proof of \cref{cor:discretencstable} is now effortless. 

\proof[Proof of \cref{cor:discretencstable}] Let $\mu$ be a probability measure on $\{0,2,3, \ldots\}$ satisfying the assumptions of \cref{cor:discretencstable}. By \cref{prop:GW}, we know that $\phi( \mathcal{D}_{n}^\mu)$ is a $ \mathsf{GW}_{ \mu}$ tree conditioned on having $n$ leaves. Set
$$B_{n}= \frac{n^{1/ \alpha}}{(| \Gamma(1- \alpha)| \cdot \mu_{0} \cdot c)^{1/ \alpha}}.$$
By \cite[Theorem 6.1 and Remark 5.10]{Kor12}, we have \begin{equation}
\label{eq:cvleaves}\left( \frac{1}{B_{n}} \cdot \W_{ \lfloor n t/ \mu_{0} \rfloor}( \tau_{n}) ; \quad 0 \leq t \leq 1\right) \quad \xrightarrow[n\to\infty]{(d)} \quad \left(X_{t}^{ \mathrm{exc},(\alpha)}; \quad 0 \leq t \leq 1 \right),
\end{equation} and, in addition, by \cite[Theorem 5.9 $(ii)$ and Remark 5.10]{Kor12}, $B_{n}/n \cdot H( \tau_{n})$ converges in distribution towards a positive real valued random variable as $ n \rightarrow \infty$, which implies that $ \mathsf{H}( \tau_{n}) /B_{n}$ converges in probability to $0$ as $n \rightarrow \infty$ since $B_{n}^{2}/n \rightarrow 0$. We are thus in position to apply \cref{thm:continuity} and get that ${B_{n}^{-1}} \cdot \mathsf{Loop}( \tau_{n})$ converges in distribution towards $ \mathscr{L}_{\alpha}$ for the Gromov--Hausdorff topology. 

We now claim that the Gromov--Hausdorff distance between $\mathcal{D}_{n}^\mu$ and $ \mathsf{Loop}( \tau_{n})$ is {roughly} bounded by the height of $ \tau_{n}$, more precisely
 \begin{eqnarray} \label{eq:closeGH} \mathrm{d_{GH}}(\mathcal{D}_{n}^\mu, \mathsf{Loop}( \tau_{n})) &\leq&    \mathsf{H}( \tau_{n}) + 2  \end{eqnarray}
{for every $n \geq 1$.} Clearly, since $ \mathsf{H}( \tau_{n}) /B_{n}$ converges in probability to $0$ as $n \rightarrow \infty$ as we have already seen, this implies the statement of the theorem. To establish \eqref{eq:closeGH}, we construct a correspondence between ${\mathcal{D}}_{n}^\mu$ and $\mathsf{Loop}( \tau_{n})$ as suggested by \cref{fig:loopdissec}: 
a point $x \in{\mathcal{D}}_{n}^\mu$ is in correspondence with a point $a \in\mathsf{Loop}( \tau_{n})$ if there exists an edge of $ \mathcal{D}_{n}^\mu$ containing both $a$ and $x$. 
\begin{figure}[!h]
 \begin{center}
 \includegraphics[width= 0.65 \linewidth]{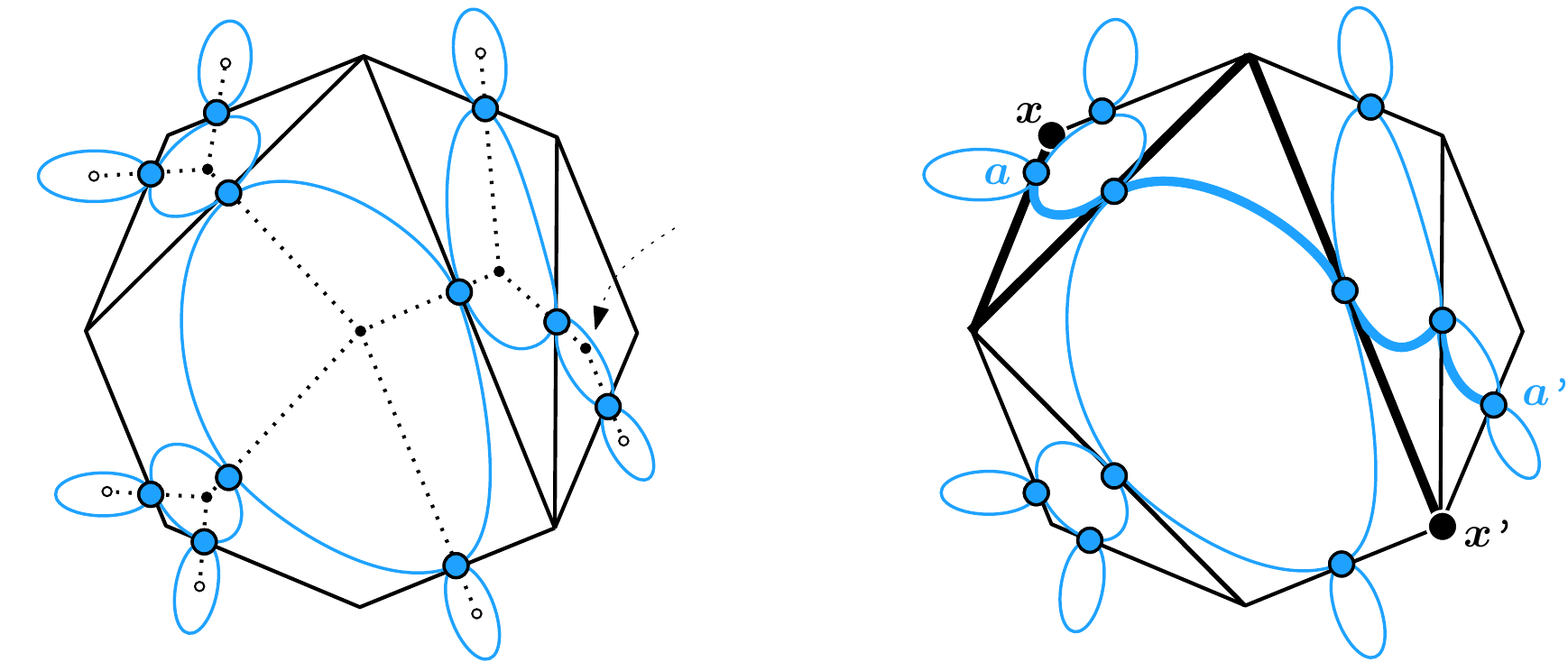}
 \caption{Close relationship between $ {\mathcal{D}}_{n}^\mu$ and $\mathsf{Loop}( \tau_{n})$. In the right-hand side figure, the geodesics $ \gamma_{a,a'}$ and $ \Gamma_{x,x'}$ are in bold.}\label{fig:loopdissec}
 \end{center}
 \end{figure}
 
 This clearly defines a correspondence between ${\mathcal{D}}_{n}^\mu$ and $\mathsf{Loop}( \tau_{n})$. Let us bound its distortion. Let $a,a' \in\mathsf{Loop}( \tau_{n})$ and $x,x'  \in {\mathcal{D}}_{n}^\mu$
 be such that $(a, x) \in \mathcal{R}$ and $(a', x') \in \mathcal{R} $. {Consider} a geodesic $\gamma_{a,a'}$ in $\mathsf{Loop}( \tau_{n})$ from $a$ to $a'$. {One can then construct a geodesic $ \Gamma_{x,x'}$ going from $x$ to $x'$ which stays ``close'' to $ \gamma_{a,a'}$ (see \cref{fig:loopdissec}), meaning that the length of the portion of $\gamma_{a,a'}$ belonging to any loop differs at most by one from the length of the portion of $ \Gamma_{x,x'}$ belonging to the corresponding face.} Since the number of loops crossed by $ \gamma_{a,a'}$ is bounded by the height of $ \tau_{n}$, it follows that $$ | \textsf{Length}(\gamma_{a,a'}) - \textsf{Length}(\Gamma_{x,x'})| \leq  \mathsf{H}( \tau_{n})+2,$$
 the term $+2$ taking into account the boundary effect due to the root edge. This yields \eqref{eq:closeGH} and finishes the proof of the corollary.
\endproof

\cref{cor:discretencstable} remains true under the more general assumption that $\mu( [k, \infty)) = L(k) \cdot k^{- \alpha}$, where $L$ is a slowly varying function at infinity. In this case, the scaling factors are slightly modified.  
\begin{remark} \label {rem:reroot}By using the fact that the law of $\mathcal{D}_{n}^\mu$ is invariant under rotations of angle  $ 2 \pi \mathbb{Z} / (n+1)$ and passing to the limit using \eqref{eq:cvleaves}, it is possible to obtain a re-rooting invariance property for looptrees, and in particular get that if $U$ and $V$ are two independent random variables uniformly distributed over  $[0,1]$, independent of $ \X$, then
$$\frac{d(U,V)}{  \mathrm{d}_{ X^{\exc}}(U,V)}  \quad\mathop{=}^{(d)} \quad \frac{d(0,U)}{ X^{\exc}(U)} .$$
\end{remark}

%%                                                               %%
%% Use the two commands below for producing your bibliography    %%
%% with bibtex, then comment again the commands and include the  %%
%% content of the .bbl file in this file below the commands.     %%
%%                                                               %%

%\bibliographystyle{amsplain}
%\bibliography{yourbibfilename}

% add below the content of your .bbl file produced by bibtex.

%%                                                               %%
%% You may add acknowledgments (optional).                       %%
%%                                                               %%

\ACKNO{We are indebted to Jean Bertoin, Lo•c Chaumont and Thomas Duquesne for several enlightening discussions on LŽvy processes and stable trees. We are also grateful to an anonymous referee for several useful comments.}

%%                                                               %%
%% You have reached the end of your document.                    %%
%%                                                               %%

\end{document}